\documentclass[a4paper,12pt,reqno]{amsart}
\usepackage[utf8]{inputenc}
\usepackage[american]{babel}
\usepackage{amsmath,amssymb,amsthm,amscd}
\usepackage{mathrsfs}
\usepackage[active]{srcltx}
\usepackage{mathrsfs}
\usepackage{enumerate}
\usepackage{lmodern}
\usepackage{braket}
\usepackage{paralist}
\usepackage[colorlinks,linkcolor={black},citecolor={black},urlcolor={black}]{hyperref}
\usepackage[margin=2.6cm,bmargin=3cm,tmargin=3cm]{geometry}
\newcommand{\supp}{\operatorname{supp}}

\newcommand{\A}{\mathcal{A}}
\newcommand{\cC}{\mathcal{C}} 
\newcommand{\cD}{\mathcal{D}} 

\newcommand{\W}{\mathcal{W}}

\newcommand{\cP}{P}
\newcommand{\cI}{\mathcal{I}}
\newcommand{\cL}{\mathscr{L}}

\newcommand{\cH}{\mathcal{H}}
\newcommand{\cJ}{\mathcal{J}}
\newcommand{\cA}{\mathcal{A}}
\newcommand{\cM}{\mathcal{M}}
\newcommand{\cN}{\mathcal{N}}

\newcommand{\cW}{\mathcal{W}}

\newcommand{\GE}{\mathrm{GE}}
\newcommand{\CGE}{\mathrm{CGE}}
\newcommand{\MLSI}{\mathrm{MLSI}}
\newcommand{\CLSI}{\mathrm{CLSI}}
\newcommand{\id}{\mathrm{id}}

\newcommand{\Ent}{\mathrm{Ent}}

\newcommand{\IC}{\mathbb{C}}
\newcommand{\IG}{\mathbb{G}}
\newcommand{\IR}{\mathbb{R}}
\newcommand{\IZ}{\mathbb{Z}}
\newcommand{\IN}{\mathbb{N}}

\newcommand{\Pol}{\mathrm{Pol}}
\newcommand{\Irr}{\mathrm{Irr}}

\newcommand{\norm}[1]{\| {#1}\|}
\newcommand{\abs}[1]{\lvert#1\rvert}

\newcommand{\ran}{\operatorname{ran}}

\newcommand{\F}{\mathcal{F}}

\theoremstyle{plain}
\newtheorem{theorem}{Theorem}[section]
\newtheorem{corollary}[theorem]{Corollary}
\newtheorem{lemma}[theorem]{Lemma}
\newtheorem{proposition}[theorem]{Proposition}

\newtheorem{definition}[theorem]{Definition}

\theoremstyle{remark}
\newtheorem{remark}[theorem]{Remark}
\newtheorem{example}[theorem]{Example}

\newtheorem*{claim*}{Claim}
\newtheorem*{remark*}{Remark}
\newtheorem*{example*}{Example}
\newtheorem*{notation*}{Notation}
\numberwithin{equation}{section}

\begin{document}

\title{Complete gradient estimates of quantum Markov semi\-groups}

\author{Melchior Wirth \and Haonan Zhang}
\address{Institute of Science and Technology Austria (IST Austria),
Am Campus 1, 3400 Klosterneuburg, Austria}
\email{melchior.wirth@ist.ac.at}
\email{haonan.zhang@ist.ac.at}
%\and Haonan Zhang \at Institute of Science and Technology Austria (IST Austria),
%Am Campus 1, 3400 Klosterneuburg, Austria\\ \email{haonan.zhang@ist.ac.at}}

%\thanks{}

\begin{abstract}
In this article we introduce a complete gradient estimate for symmetric quantum Markov semigroups on von Neumann algebras equipped with a normal faithful tracial state, which implies semi-convexity of the entropy with respect to the recently introduced noncommutative $2$-Wasserstein distance. We show that this complete gradient estimate is stable under tensor products and free products and establish its validity for a number of examples. As an application we prove a complete modified logarithmic Sobolev inequality with optimal constant for Poisson-type semigroups on free group factors.
\end{abstract}

\maketitle

\section{Introduction}

In the last decades, the theory of optimal transport has made impressive inroads into other disciplines of mathematics, probably most notably with the Lott--Sturm--Villani theory \cite{LV09,Stu06a,Stu06b} of synthetic Ricci curvature bounds for metric measure spaces. More recently, optimal transport techniques have also been used to extend this theory to cover also discrete \cite{CHLZ12,EM12,Maa11,Mie11} and noncommutative geometries \cite{CM14,CM17a,MM17}. 

The starting point of our investigation are the results from \cite{CM14,CM17a} and their partial generalizations to the infinite-dimensional case in \cite{Hor18,Wir18}. For a symmetric quantum Markov semigroup $(\cP_t)$ the authors construct a noncommutative version of the $2$-Wasserstein metric, which allows to obtain a quantum analog of the characterization \cite{JKO98,Ott01} of the heat flow as $2$-Wasserstein gradient flow of the entropy. On this basis, the geodesic semi-convexity of the entropy in noncommutative $2$-Wasserstein space can be understood as a lower Ricci curvature bound for the quantum Markov semigroup, and it can be used to obtain a series of prominent functional inequalities such as a Talagrand inequality, a modified logarithmic Sobolev inequality and the Poincaré inequality \cite{BGJ20,CM17a,JLL20,RD19}.

One of the major challenges in the development of this program so far has been to verify semi-convexity in concrete examples, and only few noncommutative examples have been known to date, even less infinite-dimensional ones.

To prove geodesic semi-convexity, the gradient estimate
\begin{equation}\label{eq:GE}
\norm{\partial \cP_t a}_\rho^2\leq e^{-2Kt}\norm{\partial a}_{\cP_t \rho}^2,\tag{GE}
\end{equation}
or, equivalently, its integrated form has proven central. They can be seen as noncommutative analogs of the Bakry--Émery gradient estimate and the $\Gamma_2$ criterion. Indeed, if the underlying quantum Markov semigroup is the heat semigroup on a complete Riemannian manifold, (\ref{eq:GE}) reduces to the classical Bakry--Émery estimate
\begin{equation*}
\Gamma(\cP_t f)\leq e^{-2Kt}\cP_t \Gamma(f).
\end{equation*}

As often in noncommutative geometry, tensorization of inequalities is more difficult than that in the commutative case. It is not known whether the gradient estimate in the form (\ref{eq:GE}) has good tensorization properties. For this reason we introduce $(\CGE)$, a complete version of (\ref{eq:GE}), and establish some of its stability properties. Using these in combination with a variant of the intertwining technique from \cite{CM17a} and a fine analysis of specific generators of Lindblad type, we are able to establish this tensor stable gradient estimate $(\CGE)$ for a number of examples for which geodesic convexity was not known before.

Let us briefly outline the content of the individual sections of this article. In Section \ref{sec:basics} we recall some basics of quantum Markov semigroups, the construction of the noncommutative transport distance $\cW$ and the connection between the gradient estimate (\ref{eq:GE}) and the geodesic semi-convexity of the entropy.

In Section \ref{sec:intertwining} we extend the intertwining technique from \cite{CM17a,CM20} to the infinite-dimensional setting. Working with arbitrary operator means, our result does not only cover the gradient estimate implying semi-convexity of the entropy in noncommutative $2$-Wasserstein space, but also the noncommutative Bakry--Émery estimate studied in \cite{JZ15a}. As  examples we show that the Ornstein--Uhlenbeck semigroup on the mixed $q$-Gaussian algebras satisfies ($\CGE$) with constant $K=1$, the heat semigroup on quantum tori satisfies ($\CGE$) with constant $K=0$, and that a class of quantum Markov semigroups on discrete group von Neumann algebras and quantum groups $O_N^+,S_N^+$ satisfy ($\CGE$) with constant $K=0$. Moreover, this intertwining result is also central for the stability properties studied in the next section.

In Section \ref{sec:stability} we show that the complete gradient estimate  is stable under tensor products and free products of quantum Markov semigroups. Besides the applications investigated later in the article, these results also open the door for applications of group transference to get complete gradient estimates for Lindblad generators on matrix algebras.

In Section \ref{sec:com_proj} we prove the complete gradient estimate ($\CGE$) with constant $K=1$ for quantum Markov semigroups whose generators are of the form 
\begin{equation*}
\cL x=\sum_j p_j x+xp_j -2p_j x p_j, 
\end{equation*}
where the operators $p_j$ are commuting projections. In a number of cases, this result is better than the ones we could obtain by intertwining and yields the optimal constant in the gradient estimate. As examples we show that this result applies to the quantum Markov semigroups associated with the word length function on finite cyclic groups and the non-normalized Hamming length function on symmetric groups. Using ultraproducts and the stability under free products, we finally extend this result to Poisson-type quantum Markov semigroups on group von Neumann algebras of groups $\IZ^{\ast k}\ast \IZ_2^{\ast l}$ with $k,l\geq 0$. In particular, this implies the complete modified logarithmic Sobolev inequality with optimal constant for these groups.

\subsection*{Note added.} When preparing this preprint for submission, we were made aware that several of the examples have been obtained independently by Brannan, Gao and Junge (see \cite{BGJ20,BGJ20b}).

\subsection*{Acknowledgments} H.Z. is supported by the European Union's Horizon 2020 research and innovation programme under the Marie Sk\l odowska-Curie grant agreement No. 754411. M.W. was supported by the Austrian Science Fund (FWF) through grant number F65. Both authors would like to thank Jan Maas for fruitful discussions and helpful comments.

\section{\texorpdfstring{The noncommutative transport metric $\W$ and geodesic convexity of the entropy}{The noncommutative transport metric W and geodesic convexity of the entropy}}\label{sec:basics}

In this section we briefly recall the definition and basic properties of the noncommutative transport distance $\W$ associated with a tracially symmetric quantum Markov semigroup. For a more detailed description we refer readers to \cite{Wir18}.

Let $\cM$ be a separable von Neumann algebra equipped with a normal faithful tracial state $\tau\colon \cM\to \IC$. Denote by $\mathcal{M}_+$ the positive cone of $\mathcal{M}$. Given $1\le p<\infty$, we define
$$\Vert x\Vert_p=[\tau(|x|^p)]^{\frac{1}{p}},~~x\in\mathcal{M},$$
where $|x|=(x^*x)^{\frac{1}{2}}$ is the modulus of $x$. One can show that $\|\cdot \|_p$ is a norm on $\mathcal{M}$.
The completion of $( \mathcal{M}, \|\cdot \|_p )$ is denoted by $L^p (\mathcal{M}, \tau)$, or simply $L^p (\mathcal{M})$. As usual, we put $L^\infty(\cM)=\cM$ with the operator norm. In this article, we are only interested in $p=1$ and $p=2$. In particular, $L^2(\cM)$ is a Hilbert space with the inner product
$$\langle x,y\rangle_2=\tau(x^*y).$$

%The elements of $L^p(\mathcal{M})$ can be viewed as closed densely defined operators on $H$ ($H$ being the Hilbert space on which $\mathcal{M}$ acts). 
%A linear closed operator $x$ on $H$ ($H$ being the Hilbert space on which $\mathcal{M}$ acts) is said to be \textit{affiliated with} $\mathcal{M}$ if it commutes with all unitary elements in $\mathcal{M}'$, i.e., $xu=ux$ for any unitary $u\in\mathcal{M}'$, where $\mathcal{M}'$ denotes the commutant of $\mathcal{M}$. Note that $x$ can be unbounded on $H$.

 A family $(\cP_t)_{t\geq 0}$ of bounded linear operators on $\cM$ is called a \emph{quantum Markov semigroup (QMS)} if 
\begin{itemize}
\item $\cP_t$ is a normal unital completely positive map for every $t\geq 0$,
\item $\cP_s \cP_t=\cP_{s+t}$ for all $s,t\geq 0$,
\item $\cP_t x\to x$ in the weak$^\ast$ topology as $t\searrow 0$ for every $x\in \cM$.
\end{itemize}
A QMS $(\cP_t)$ is called \emph{$\tau$-symmetric} if
\begin{equation*}
\tau((\cP_t x)y)=\tau(x \cP_t y)
\end{equation*}
for all $x,y\in \cM$ and $t\geq 0$.

The generator of $(\cP_t)$ is the operator $\cL$ given by
\begin{align*}
D(\cL)&=\left\{x\in \cM\mid \lim_{t\searrow 0}\frac1 t (x-\cP_t x)\text{ exists in the weak$^\ast$ topology}\right\},\\
\cL(x)&=\lim_{t\to 0}\frac1 t (x-\cP_t x),~~x\in D(\cL).
\end{align*}
Here and in what follows, $D(T)$ always means the domain of $T$.
For all $p\in [1,\infty]$, the $\tau$-symmetric QMS $(\cP_t)$ extends to a strongly continuous contraction semigroup $(\cP_t^{(p)})$ on $L^p(\cM,\tau)$ with generator $\cL_p$.

Let $\cC=D(\cL_2^{1/2})\cap \cM$, which is a $\sigma$-weakly dense $\ast$-subalgebra of $\cM$ \cite[Proposition 3.4]{DL92}. According to \cite[Section 8]{CS03}, there exists an (essentially unique) quintuple $(\cH, L, R, \cJ, \partial)$ consisting of a Hilbert space $\cH$, commuting non-degenerate $^\ast$-homomorphisms $L\colon \cC\to B(\cH)$, $R\colon \cC^\circ\to B(\cH)$, an anti-unitary involution $\cJ\colon \cH\to \cH$ and a closed operator $\partial\colon D(\cL_2^{1/2})\to \cH$ such that
\begin{itemize}
\item $\{R(x)\partial y\mid x,y\in \cC\}$ is dense in $\cH$,
\item $\partial(xy)=L(x)\partial y+R(y)\partial x$ for $x,y\in \cC$,
\item $\cJ(L(x)R(y)\partial(z))=L(y^\ast)R(x^\ast)\partial(z^\ast)$ for $x,y,z\in \cC$,
\item $\cL_2=\partial^\dagger \partial$, 
\end{itemize}
where $\cC^\circ$ is the opposite algebra of $\cC$. We will write $a\xi$ and $\xi b$ for $L(a)\xi$ and $R(b)\xi$, respectively.
%
%\begin{assumption}
%We assume that there exists a $\sigma$-weakly dense unital $\ast$-algebra $\cA$ of $D(\cL_2)\cap \cM$ that is invariant under $(\cP_t)$.
%\end{assumption}
%While not strictly necessary, this assumption is convenient to do away with a more technical condition in \cite{Wir18} and covers all the examples we have in mind. The same or similar assumptions have been made in \cite{JM12,Cas18,CIW20}.

For $x,y\in D(\cL_2^{1/2})$, the \emph{carré du champ} is defined as
\begin{equation*}
\Gamma(x,y)\colon \cC\to\IC,\,\Gamma(x,y)(z)=\langle \partial x,(\partial y)z\rangle_\cH.
\end{equation*}
We write $\Gamma(x)$ to denote $\Gamma(x,x)$.

A $\tau$-symmetric QMS is called \emph{$\Gamma$-regular} (see \cite{JLL20}) if the representations $L$ and $R$ are normal.
Under this assumption, $\cH$ is a \emph{correspondence} from $\cM$ to itself in the sense of Connes \cite[Appendix B of Chapter 5]{Con94} (sometimes also called \emph{$\cM$-bimodule} or \emph{Hilbert bimodule}). By \cite[Theorem 2.4]{Wir18}, $(\cP_t)$ is a $\Gamma$-regular semigroup if and only if $\Gamma(x,y)$ extends to a normal linear functional on $\cM$ for all $x,y\in D(\cL^{1/2}_2)$. By a slight abuse of notation, we write $\Gamma(x,y)$ for the unique element $h\in L^1(\cM,\tau)$ such that
\begin{align*}
\tau(z h)=\Gamma(x,y)(z)
\end{align*}
for all $z\in \cC$.

If $(\cP_t)$ is $\Gamma$-regular, then we can extend $L$ to a map on the operators affiliated with $\cM$ by defining
\begin{align*}
L(x)=u\int_{[0,\infty)}\lambda\,d(L\circ e)(\lambda),
\end{align*}
for any operator $x$ affiliated with $\cM$, where $u$ is the partial isometry in the polar decomposition of $x$ and $e$ is the spectral measure of $\abs{x}$. The same goes for $R$.

Let $\Lambda$ be an  \emph{operator mean} in the sense of Kubo and Ando \cite{KA80}, that is, $\Lambda$ is a map from $B(\cH)_+\times B(\cH)_+$ to $B(\cH)_+$ satisfying 
\begin{enumerate}[(a)]
\item if $A\leq C$ and $B\leq D$, then $\Lambda(A,B)\leq\Lambda(C,D)$,
\item the \emph{transformer inequality}: $C \Lambda(A,B)C\leq \Lambda(C A C,C B C)$ for all $A,B,C\in B(\cH)_+$,
\item if $A_n\searrow A$, $B_n\searrow B$, then $\Lambda(A_n,B_n)\searrow \Lambda(A,B)$,
\item $\Lambda(\id_{\cH},\id_{\cH})=\id_{\cH}$.
\end{enumerate}
Here and in what follows, by $A_n\searrow A$ we mean $A_1\ge A_2\ge \cdots$ and  $A_n$ converges strongly to $A$. From (b), any operator mean $\Lambda$ is \emph{positively homogeneous}:
$$\Lambda(\lambda A,\lambda B)=\lambda\Lambda (A,B),~~\lambda >0,A,B\in B(\cH)_+.$$
An  operator mean $\Lambda$ is \emph{symmetric} if $\Lambda(A,B)=\Lambda(B,A)$ for all $A,B\in B(\cH)_+$. %Remark that in the transformer inequality (b) is also valid when $C$ is not self-adjoint \cite[Theorem 3.5]{KA80}.

For a positive self-adjoint operator $\rho$ affiliated with $\cM$, we define
\begin{equation*}
\hat \rho=\Lambda(L(\rho), R(\rho)).
\end{equation*}
Of particular interest for us are the cases when $\Lambda$ is the \emph{logarithmic mean}
\begin{equation*}
\Lambda_{\text{log}}(L(\rho), R(\rho))=\int_{0}^{1}L(\rho)^s R(\rho)^{1-s}\,ds,
\end{equation*}
or the \emph{arithmetic mean}
\begin{equation*}
\Lambda_{\text{ari}}(L(\rho), R(\rho))=\frac{L(\rho)+R(\rho)}{2}.
\end{equation*}

We write $\norm{\cdot}_\rho^2$ for the quadratic form associated with $\hat \rho$, that is,
\begin{align*}
\norm{\xi}_\rho^2=\begin{cases}\norm{\hat \rho^{1/2}\xi}_\cH^2&\text{if }\xi\in D(\hat \rho^{1/2}),\\
\infty&\text{otherwise}.\end{cases}
\end{align*}
Given an operator mean $\Lambda$, consider the set
\begin{align*}
\cA_\Lambda=\{a\in D(\cL_2^{1/2})\cap\cM\mid \exists C>0\,\forall \rho\in L^1_+(\cM,\tau)\colon \norm{\partial a}_\rho^2\leq C\norm{\rho}_1\},
\end{align*}
equipped with the seminorm
\begin{align*}
\norm{a}_\Lambda^2=\sup_{0\neq \rho\in L^1_+(\cM,\tau)}\frac{\norm{\partial a}_\rho^2}{\norm{\rho}_1}.
\end{align*}
If $\Lambda$ is the arithmetic mean $\Lambda_{\text{ari}}$, then this set coincides with
\begin{align*}
\cA_\Gamma=\{x\in D(\cL_2^{1/2})\cap\cM\mid \Gamma(x),\Gamma(x^\ast)\in \cM\}.
\end{align*}
In fact, when $\Lambda=\Lambda_{\text{ari}}$, one has $\norm{\partial a}_{\rho}^2=\frac{1}{2}\tau\left((\Gamma(a)+\Gamma(a^*))\rho\right)$. 
If the operator mean $\Lambda$ is symmetric, then it is dominated by the arithmetic mean and therefore $\cA_\Gamma\subset \cA_\Lambda $ \cite[Theorem 4.5]{KA80},\cite[Lemma 3.24]{Wir18}. The following definition states that this inclusion is dense in a suitable sense.

\begin{definition}\label{def:regular_mean}
The operator mean $\Lambda$ is a \emph{regular mean} for $(\cP_t)$ if for every $x\in \cA_\Lambda$ there exists a sequence $(x_n)$ in $\cA_\Gamma$ that is bounded in $\A_\Lambda$ and converges to $x$ $\sigma$-weakly.
\end{definition}
Of course the arithmetic mean is always regular. In general it seems not easy to check this definition directly, but we will discuss a sufficient condition below.

Given an operator mean $\Lambda$, let $\cH_\rho$ be the Hilbert space obtained from $\partial(\cA_\Lambda)$ after separation and completion with respect to $\langle\cdot,\cdot\rangle_\rho$ defined by 
\begin{equation*}
\langle\xi,\eta\rangle_\rho=\langle\hat\rho^{1/2}\xi,\hat{\rho}^{1/2}\eta\rangle_{\cH}.
\end{equation*}
 If $\Lambda$ is regular, then $\partial(\cA_\Gamma)$ is dense in $\cH_\rho$.

Let $\cD(\cM,\tau)$ be the set of all \emph{density operators}, that is,
\begin{equation*}
\cD(\cM,\tau)=\{\rho\in L^1_+(\cM,\tau)\mid \tau(\rho)=1\}.
\end{equation*}

\begin{definition}\label{def:admissible_curve}
Fix an operator mean $\Lambda$. A curve $(\rho_t)_{t\in [0,1]}\subset \cD(\cM,\tau)$ is \emph{admissible} if 
\begin{itemize}
\item the map $t\mapsto \tau(a\rho_t)$ is measurable for all $a\in \cA_\Gamma$,
\item there exists a curve $(\xi_t)_{t\in [0,1]}$ such that $\xi_t\in \cH_{\rho_t}$ for all $t\in [0,1]$, the map $t\mapsto \langle \partial a,\xi_t\rangle_{\rho_t}$ is measurable for all $a\in \cA_\Gamma$ and for every $a\in \cA_\Gamma$ one has
\begin{equation}\label{eq:CE}
\frac{d}{dt}\tau(a\rho_t)=\langle \xi_t,\partial a\rangle_{\rho_t}
\end{equation}
for a.e. $t\in [0,1]$.
\end{itemize}
\end{definition}
For an admissible curve $(\rho_t)$, the vector field $(\xi_t)$ is uniquely determined up to equality a.e. and will be denoted by $(D\rho_t)$. If $\Lambda$ is a regular mean, the set $\cA_\Gamma$ can be replaced by $\cA_\Lambda$ everywhere in Definition \ref{def:admissible_curve}.

\begin{remark}
The equation (\ref{eq:CE}) is a weak formulation of 
\begin{equation*}
\dot\rho_t=\partial^\dagger (\hat\rho_t \xi_t),
\end{equation*}
which can be understood as noncommutative version of the continuity equation. Indeed, if $(\cP_t)$ is the heat semigroup on a compact Riemannian manifold, it reduces to the classical continuity equation $\dot\rho_t+\operatorname{div}(\rho_t \xi_t)=0$.
\end{remark}

\begin{definition}
The noncommutative transport distance $\cW$ on $\cD(\cM,\tau)$ is defined as
\begin{equation*}
\cW(\bar\rho_0,\bar\rho_1)=\inf_{(\rho_t)}\int_0^1 \norm{D\rho_t}_{\rho_t}\,dt,
\end{equation*}
where the infimum is taken over all admissible curves $(\rho_t)$ connecting $\bar \rho_0$ and $\bar \rho_1$.
\end{definition}

\begin{definition}\label{defn:CGE}
Let $K\in \IR$. A $\Gamma$-regular QMS $(\cP_t)$ is said to satisfy the gradient estimate $\GE(K,\infty)$ if 
\begin{equation*}
\norm{\partial \cP_t a}_\rho^2\leq e^{-2Kt}\norm{\partial a}_{\cP_t \rho}^2
\end{equation*}
for $t\geq 0$, $a\in D(\cL_2^{1/2})$ and $\rho\in \cD(\cM,\tau)$.

It satisfies $\CGE(K,\infty)$ if $(\cP_t\otimes \id_{\cN})$ satisfies $\GE(K,\infty)$ for any finite von Neumann algebra $\cN$.
\end{definition}

Note that the gradient estimate $\GE(K,\infty)$ depends implicitly on the chosen operator mean $\Lambda$. As observed in \cite[Proposition 6.12]{Wir18}, if $(\cP_t)$ satisfies $\GE(K,\infty)$ for the arithmetic mean $\Lambda_{\text{ari}}$ and for $\Lambda$, then $\Lambda$ is regular for $(\cP_t)$. If $\Lambda$ is the right trivial mean, i.e., $\Lambda(L(\rho),R(\rho))=R(\rho)$, then $\GE(K,\infty)$ reduces to the Bakry--Émery criterion
\begin{equation*}
\Gamma(\cP_t a)\leq e^{-2Kt}\cP_t\Gamma(a),
\end{equation*}
which was considered in \cite{JZ15a}.

\begin{remark}
Recently, Li, Junge and LaRacuente \cite{JLL20} introduced a closely related notion of lower Ricci curvature bound for quantum Markov semigroups, the \emph{geometric Ricci curvature condition} (see also \cite[Definition 3.22]{BGJ20}). Like $\CGE$, this condition is tensor stable, and it implies $\CGE$ for arbitrary operator means \cite[Theorem 3.6]{JLL20} (the result is only formulated for the logarithmic mean, but the proof only uses the transformer inequality for operator means).

In the opposite direction, the picture is less clear. For $\GE$, a direct computation on the two-point graph shows that the optimal constant depends on the mean in general. It seems reasonable to expect the same behavior for $\CGE$, which would imply that the optimal constant in $\CGE$ for a specific mean is in general bigger than the optimal constant in the geometric Ricci curvature condition.
\end{remark}

This gradient estimate is closely related to convexity properties of the logarithmic entropy
\begin{equation*}
\Ent\colon \cD(\cM,\tau)\to [0,\infty],\,\Ent(\rho)=\tau(\rho\log \rho).
\end{equation*}
As usual let $D(\Ent)=\{\rho\in \cD(\cM,\tau)\mid \Ent(\rho)<\infty\}$.

\begin{theorem}[{\cite[Theorem 7.12]{Wir18}}]\label{thm:geodesic_convex}
Assume that $(\cP_t)$ is a $\Gamma$-regular QMS. Suppose that $\Lambda=\Lambda_{\log}$ is the logarithmic mean and is regular for $(\cP_t)$. If $(\cP_t)$ satisfies $\GE(K,\infty)$, then
\begin{enumerate}[(a)]
\item for every $\rho\in D(\Ent)$ the curve $(\cP_t \rho)$ satisfies the \emph{evolution variational inequality} $(\mathrm{EVI}_K)$
\begin{equation*}
\frac{d}{dt}\frac 1 2\cW^2(\cP_t \rho,\sigma)+\frac K 2 \cW^2(\cP_t \rho,\sigma)+\Ent(\rho)\leq \Ent(\sigma)
\end{equation*}
for a.e. $t\geq 0$ and $\sigma \in \cD(\cM,\tau)$ with $\cW(\rho,\sigma)<\infty$,
\item any $\rho_0,\rho_1\in D(\Ent)$ with $\cW(\rho_0,\rho_1)<\infty$ are connected by a $\cW$-geodesic and $\Ent$ is $K$-convex along any constant speed $\cW$-geodesic $(\rho_t)$, that is, $\frac{d^2}{dt^2}\Ent(\rho_t)\geq K$ in the sense of distributions.
\end{enumerate}
\end{theorem}

This gradient flow characterization implies a number of functional inequalities for the QMS, see e.g. \cite[Section 8]{CM17a}, \cite[Section 7]{Wir18}, \cite[Section 11]{CM20}. Here we will focus on the modified logarithmic Sobolev inequality and its complete version (see \cite[Definition 2.8]{GJL18}, \cite[Definition 2.12]{JLL20} for the latter).

For $\rho,\sigma\in \cD(\cM,\tau)$ the \emph{relative entropy} of $\rho$ with respect to $\sigma$ is defined as
\begin{equation*}
\Ent(\rho\Vert \sigma)=\begin{cases}\tau(\rho\log \rho)-\tau(\rho\log\sigma)&\text{if }\supp \rho\subset \supp \sigma,\\
\infty&\text{otherwise}.\end{cases}
\end{equation*}
If $\cN\subset \cM$ is a von Neumann subalgebra with $E\colon \cM\to \cN$ being the conditional expectation, then we define
\begin{align*}
\Ent_\cN(\rho)=\Ent(\rho\lVert E(\rho)).
\end{align*}
Recall that a \emph{conditional expectation} $E\colon\cM\to\cN$ is a normal contractive positive projection from $\cM$ onto $\cN$ which preserves the trace and satisfies
\begin{equation*}
E(axb)=aE(x)b,~~a,b\in\cN, x\in\cM.
\end{equation*}

For $x\in D(\cL_2^{1/2})\cap \cM_+$ the \emph{Fisher information} is defined as
\begin{equation*}
\cI(x)=\lim_{\epsilon\searrow 0}\langle \cL_2^{1/2} x,\cL_2^{1/2}\log(x+\epsilon)\rangle_2\in [0,\infty].
\end{equation*}
This definition can be extended to $x\in L^1_+(\cM,\tau)$ by setting
\begin{equation*}
\cI(x)=\begin{cases}\lim_{n\to\infty}\cI(x\wedge n)&\text{if }x\wedge n\in D(\cL_2^{1/2})\cap \cM\text{ for all }n\in\IN,\\
\infty&\text{otherwise}.\end{cases}
\end{equation*}

Recall that the fixed-point algebra of $(P_t)$ is
\begin{equation*}
\cM^{\mathrm{fix}}=\{x\in\cM:\cP_t(x)=x\text{ for all }t\geq 0\}.
\end{equation*}
It is a von Neumann subalgebra of $\cM$ \cite[Proposition 3.5]{DL92}.

\begin{definition}
Let $(\cP_t)$ be a $\Gamma$-regular QMS with the fixed-point subalgebra $\cM^{\mathrm{fix}}$. For $\lambda>0$, we say that $(\cP_t)$ satisfies the modified logarithmic Sobolev inequality with constant $\lambda$ $(\MLSI(\lambda))$, if
\begin{equation*}
\lambda \Ent_{\cM^{\mathrm{fix}}}(\rho)\leq \cI(\rho)
\end{equation*}
for $\rho\in\cD(\cM,\tau)\cap D(\cL_2^{1/2})\cap \cM$.

We say that $(\cP_t)$ satisfies the complete modified logarithmic Sobolev inequality with constant $\lambda$ $(\CLSI(\lambda))$ if $(\cP_t\otimes \id_\cN)$ satisfies the modified logarithmic Sobolev inequality with constant $\lambda$ for any finite von Neumann algebra $\cN$.
\end{definition}

For ergodic QMS satisfying $\GE(K,\infty)$, the inequality $\MLSI(2K)$ is essentially contained in the proof of \cite[Proposition 7.9]{Wir18}. Since $(\cP_t\otimes \id_\cN)$ is not ergodic (unless $\cN=\IC$), this result cannot imply the complete modified logarithmic Sobolev inequality. However, the modified logarithmic Sobolev inequality for non-ergodic QMS can also still be derived from the gradient flow characterization, as we will see next.

\begin{corollary}\label{cor:MLSI}
Assume that $(\cP_t)$ is a $\Gamma$-regular QMS. Suppose that $\Lambda=\Lambda_{\log}$ is the logarithmic mean and is regular for $(\cP_t)$. If $(\cP_t)$ satisfies $\GE(K,\infty)$, then it satisfies 
\begin{equation*}
\cI(\cP_t \rho)\leq e^{-2Kt}\cI(\rho)
\end{equation*}
for $\rho \in D(\cL_2^{1/2})\cap\cM_+$ and $t\geq 0$.

Moreover, if $K>0$, then $(\cP_t)$ satisfies $\MLSI(2K)$. The same is true for the complete gradient estimate and the complete modified  logarithmic Sobolev inequality.
\end{corollary}
\begin{proof}
Let $\rho\in\cD(\cM,\tau)\cap D(\cL_2^{1/2})\cap\cM$ and $\rho_t=\cP_t\rho$. Since $(\rho_t)$ is an $\mathrm{EVI}_K$ gradient flow curve of $\Ent$ by Theorem \ref{thm:geodesic_convex} and $\frac{d}{dt}\Ent(\rho_t)=-\cI(\rho_t)$, it follows from \cite[Theorem 3.5]{MS20} that
\begin{equation*}
\cI(\cP_t \rho)\leq e^{-2Kt}\cI(\rho)
\end{equation*}
for $t\geq 0$ (using the continuity of both sides in $t$).

If $K>0$, then $\MLSI(2K)$ follows from a standard argument; see for example \cite[Lemma 2.15]{JLL20}. The implication for the complete versions is clear.
\end{proof}

\begin{remark}
The inequality $\cI(\cP_t \rho)\leq e^{-2Kt}\cI(\rho)$ is called $K$-Fisher monotonicity in \cite{BGJ20} and plays a central role there in obtaining complete logarithmic Sobolev inequalities.
\end{remark}

\section{Gradient estimates through intertwining}\label{sec:intertwining}

Following the ideas from \cite{CM17a,CM20}, we will show in this section how one can obtain gradient estimates for quantum Markov semigroups through intertwining. As examples we discuss the Ornstein--Uhlenbeck semigroup on the mixed $q$-Gaussian algebras, the heat semigroup on quantum tori, and a family of quantum Markov semigroups on discrete group von Neumann algebras and the quantum groups $O_N^+$ and $S_N^+$.

Throughout this section we assume that $\cM$ is a separable von Neumann algebra with normal faithful tracial state $\tau$ and $(\cP_t)$ is a $\Gamma$-regular QMS. We fix the corresponding first order differential calculus $(\cH, L, R, \cJ, \partial)$. We do not make any assumptions on $\Lambda$ beyond being an operator mean. In particular, all results from this section apply to the logarithmic mean -- thus yielding geodesic convexity by Theorem \ref{thm:geodesic_convex} --- as well as the right-trivial mean -- thus giving Bakry--Émery estimates.

\begin{theorem}\label{thm:intertwining}
	Let $K\in\mathbb{R}$. If there exists a family $(\vec \cP_t)$ of bounded linear operators on $\cH$ such that
	\begin{enumerate}[(i)]
		\item $\partial \cP_t =\vec \cP_t \partial$ for $t\geq 0$,
		\item $\vec \cP_t^\dagger L(\rho) \vec \cP_t\leq e^{-2Kt}L(\cP_t \rho)$ for $\rho\in\cM_+$, $t\geq 0$,
		\item $\vec \cP_t^\dagger R(\rho) \vec \cP_t\leq e^{-2Kt}R(\cP_t \rho)$ for $\rho\in\cM_+$, $t\geq 0$,
	\end{enumerate}
	then $(\cP_t)$ satisfies $\GE(K,\infty)$.
\end{theorem}

\begin{proof}
	Let $\rho\in \cM_+$ and $a\in D(\partial)$. Since $\Lambda$ is an operator mean, we have \cite[Theorem 3.5]{KA80}
	\begin{equation*}
	\vec \cP_t^\dagger \Lambda(L(\rho),R(\rho))\vec \cP_t\leq \Lambda(\vec \cP_t^\dagger L(\rho)\vec \cP_t,\vec \cP_t^\dagger R(\rho) \vec \cP_t).
	\end{equation*}
	Thus
	\begin{equation*}
	\langle \hat \rho \partial \cP_t a,\partial \cP_t a\rangle_\cH=\langle \vec \cP_t^\dagger \hat \rho \vec \cP_t \partial a,\partial a\rangle_\cH\leq \langle \Lambda(\vec \cP_t^\dagger L(\rho)\vec \cP_t,\vec \cP_t^\dagger R(\rho)\vec \cP_t)\partial a,\partial a\rangle_\cH.
	\end{equation*}
	As $\Lambda$ is monotone in both arguments and positively homogeneous, conditions (ii) and (iii) imply
	\begin{equation*}
	\langle \Lambda(\vec \cP_t^\dagger L(\rho)\vec \cP_t,\vec \cP_t^\dagger R(\rho)\vec \cP_t)\partial a,\partial a\rangle_\cH\leq e^{-2Kt}\langle \Lambda(L(\cP_t \rho),R(\cP_t \rho))\partial a,\partial a\rangle_\cH.
	\end{equation*}
	All combined this yields 
	\begin{equation*}
	\norm{\partial \cP_t a}_\rho^2\leq e^{-2Kt}\norm{\partial a}_{\cP_t \rho}^2.\qedhere
	\end{equation*}
% For any finite von Neumann algebra $\cN$, the first order differential calculus associated with $\cP_t\otimes \id_{\cN}$ is given by 
%$$(\widetilde{\cH},\widetilde{L},\widetilde{R},\widetilde{\cJ},\widetilde{\partial})=(\cH\otimes L^2(\cN),L\otimes\id_{\cC},R\otimes\id_{\cC^\circ},\cJ\otimes\id_{L^2(\cN)},\partial\otimes\id_{\cN}).$$
%So the conditions (i-iii) are satisfied for $\cP_t\otimes \id_{\cN}$ with $\overrightarrow{\cP_t\otimes \id_{\cN}}=\vec \cP_t\otimes \id_{\cN}$. Following the above argument, one can show that $ \cP_t\otimes \id_{\cN}$ also satisfies $\GE(K,\infty)$. Hence $(\cP_t)$ satisfies $\CGE(K,\infty)$.
\end{proof}

\begin{remark}\label{rmk:diff_calc_ind}
The proof shows that assumptions (i)--(iii) still imply
\begin{equation*}
\norm{\partial \cP_t a}_{\rho}^2\leq e^{-2Kt}\norm{\partial a}_{\cP_t \rho}^2
\end{equation*}
if the differential calculus is not the one associated with $(\cP_t)$. We will use this observation in the proofs of Theorem \ref{thm:tensor_product} and Theorem \ref{thm:com_proj}.
\end{remark}

\begin{remark}
A similar technique to obtain geodesic convexity of the entropy has been employed in \cite{CM17a,CM20}. Our proof using the transformer inequality for operator means is in some sense dual to the monotonicity argument used there (see \cite{Pet96}). Apart from working in the infinite-dimensional setting, let us point out two main differences to the results from these two articles:

In contrast to \cite{CM17a}, we do not assume that $\vec \cP_t$ is a direct sum of copies of $\cP_t$ (in fact, we do not even assume that $\cH$ is a direct sum of copies of the trivial bimodule). This enhanced flexibility can lead to better bounds even for finite-dimensional examples (see Example \ref{ex:cond_exp}). In contrast to \cite{CM20}, our conditions (ii) and (iii) are more restrictive, but they are also linear in $\rho$, which makes them potentially more feasible to check in concrete examples.
\end{remark}

\begin{remark}
We do not assume that the operators $\vec \cP_t$ form a semigroup or that they are completely positive (if $\cH$ is realized as a subspace of $L^2(\cN)$ for some von Neumann algebra $\cN$). However, this is the case for most of the concrete examples where we can prove (i)--(iii).
\end{remark}

\begin{remark}
In particular, the conclusion of the previous theorem holds for all symmetric operator means, and in view of the discussions after Definition \ref{defn:CGE}, it implies that any symmetric operator mean is regular for $(P_t)$.
\end{remark}

Under a slightly stronger assumption, conditions (ii) and (iii) can be rewritten in a way that resembles the classical Bakry--Émery criterion. For that purpose define %(see \cite[Section 2]{Wir18})
\begin{equation*}
\vec \Gamma\left(\sum_{k=1}^n (\partial x_k)y_k\right)=\sum_{k,l=1}^n y_k^\ast \Gamma(x_k,x_l)y_l.
\end{equation*}
In particular, $\vec \Gamma(\partial x)=\Gamma(x)$. Since $(\cP_t)$ is $\Gamma$-regular, $\vec \Gamma$ extends to a continuous quadratic map from $\cH$ to $L^1(\cM,\tau)$ that is uniquely determined by the property $\tau(x\vec \Gamma(\xi))=\langle\xi,\xi x\rangle_\cH$ for all $x\in \cM$ and $\xi\in \cH$ (see \cite[Section 2]{Wir18}).

\begin{lemma}
If $(\vec \cP_t)$ is a family of bounded linear operators on $\cH$ that commute with $\cJ$, then conditions (ii) and (iii) from Theorem \ref{thm:intertwining} are equivalent. Moreover, they are equivalent to
\begin{equation}\label{eq:vec_Bakry_Emery}
\vec \Gamma(\vec \cP_t \xi)\leq e^{-2Kt}\cP_t \vec\Gamma(\xi)
\end{equation}
for $\xi\in \cH$, $t\geq 0$.
\end{lemma}
\begin{proof}
To see the equivalence of (ii) and (iii), it suffices to notice that $\cJ$ is a bijection and $\cJ L(\rho)\cJ=R(\rho)$ for $\rho\in \cM_+$. The equivalence of (iii) and (\ref{eq:vec_Bakry_Emery}) follows from the identities: for all $\rho\in \cM_+$:
\begin{align*}
\langle \vec \cP_t \xi,R(\rho)\vec \cP_t \xi\rangle_\cH&=\tau(\rho \vec \Gamma(\vec \cP_t \xi)),\\
\langle \xi,R(\cP_t \rho)\xi\rangle_\cH&=\tau(\cP_t \rho\vec \Gamma(\xi))=\tau(\rho \cP_t \vec \Gamma(\xi)).\qedhere
\end{align*}
\end{proof}

As indicated before, our theorem recovers the intertwining result in \cite{CM17a} (in the tracially symmetric case):
\begin{corollary}\label{cor:intertwining_dir_sum}
Assume that $\cH\cong\bigoplus_j L^2(\cM,\tau)$, $L$ and $R$ act componentwise as left and right multiplication and $\cJ$ acts componentwise as the usual involution. If $\partial_j \cP_t=e^{-Kt}\cP_t \partial_j$, then $(\cP_t)$ satisfies $\CGE(K,\infty)$.
\end{corollary}
\begin{proof}
Let $\vec \cP_t=e^{-Kt}\bigoplus_j \cP_t$. Condition (i) from Theorem \ref{thm:intertwining} is satisfied by assumption. Since $\vec \cP_t$ commutes with $\cJ$, conditions (ii) and (iii) are equivalent. Condition (iii) follows directly from the Kadison--Schwarz inequality:
\begin{align*}
\langle \vec \cP_t\xi,R(\rho)\vec \cP_t\xi\rangle_\cH&=\sum_{j}e^{-2Kt}\tau((\cP_t \xi_j)^\ast (\cP_t \xi_j)\rho)\\
&\leq e^{-2Kt}\sum_j \tau(\xi_j^\ast \xi_j \cP_t\rho)\\
&=e^{-2Kt}\langle \xi,R(\cP_t \rho)\xi\rangle_\cH.
\end{align*}
This settles $\GE(K,\infty)$. Applying the same argument to $(\cP_t\otimes \id_\cN)$ then yields the complete gradient estimate.
\end{proof}

\begin{example}[Conditional expectations]\label{ex:cond_exp}
Let $E\colon \cM\to \cN$ be the conditional expectation onto a von Neumann subalgebra $\cN$ and let $(\cP_t)$ be the QMS with generator $\cL=I-E$, where $I=\id_{\cM}$ is the identity operator on $\cM$. Then $(\cP_t)$ satisfies $\CGE(1/2,\infty)$:

A direct computation shows that $\cP_t =e^{-t}I+(1-e^{-t})E$. Let $\vec \cP_t=e^{-t}\id_{\cH}$. Since $\cL E=0$, we have $\partial E=0$ and therefore $\partial \cP_t =e^{-t}\partial=\vec{P}_t\partial$, which settles condition (i) from Theorem \ref{thm:intertwining}. Conditions (ii) and (iii) with $K=1/2$ follow immediately from $\cP_t \rho\geq e^{-t}\rho$ for $\rho\in \cM_+$. So $(\cP_t)$ satisfies $\CGE(1/2,\infty)$. This result has been independently obtained in \cite[Theorem 4.16]{BGJ20}.
%Finally, the generator of $(\cP_t\otimes \id_\cN)$ is of the same form, so that it also satisfies $\GE(1/2,\infty)$.

In contrast, if for example $p$ is a projection and $E(x)=pxp+(1-p)x(1-p)$, then $\cL$ has the Lindblad form $\cL x=[p,[p,x]]$. Clearly, $[p,\cdot]$ commutes with $\cL$, so that the intertwining criterion from \cite{CM17a} only implies $\CGE(0,\infty)$. In fact, in this case we may obtain a better result; see Theorem \ref{thm:com_proj}.
\end{example}

\begin{example}[Mixed $q$-Gaussian algebras]
	Let us recall the mixed $q$-Gaussian algebras. Our references are \cite{BS91,BS94,BKS97,LP99}. Let $H$ be a real Hilbert space with orthonormal basis $(e_j)_{j\in J}$.  For $k\ge 1$, denote by $S_k$ the set of permutations of $\{1,2,\dots,k\}$. For $k\ge 2$ and $1\le j\le k-1$, denote by $\sigma_{j}$ the adjacent transposition between $j$ and $j+1$. For any $\sigma\in S_k$, $I(\sigma)$ is the number of inversions of the permutations $\sigma$: 
	$$I(\sigma)=\sharp\{(i,j):1\le i<j\le k,~\sigma(i)>\sigma(j)\}.$$
	For $k\ge 1$, a \emph{$k$-atom} on $H$ is an element of the form $f_1\otimes \cdots \otimes  f_k$ with each $f_j\in H$. A \emph{$k$-basis atom} is an element of the form $e_{j_1}\otimes \cdots\otimes e_{j_k}$. Clearly all the $k$-basis atoms form a basis of $H^{\otimes k}$. For any $k$-basis atom $u=e_{j_1}\otimes \cdots\otimes e_{j_k}$, we use the notation that $\sigma(u)=e_{j_{\sigma(1)}}\otimes \cdots\otimes e_{j_{\sigma(k)}}$. 
	
	Let $Q=(q_{ij})_{i,j\in J}\in\IR^{J\times J}$ be such that  $q_{ij}=q_{ji}$ for all $i,j\in J$ and $\sup_{i,j\in J}|q_{ij}|\le1$. 
  For convenience, in the following we actually assume that $\sup_{i,j\in J}|q_{ij}|<1$. This is to simplify the definition of Fock space; our main results still apply to the general $\sup_{i,j\in J}|q_{ij}|\le1$ case.
	
	 Put $P^{(0)}=\id_{H}$. For any $k\ge 1$, denote by $P^{(k)}$ the linear operator on $H^{\otimes k}$ such that 
	\begin{equation*}
	P^{(k)}(u)=\sum_{\sigma\in S_k}a(Q,\sigma,u)\sigma^{-1}(u),
	\end{equation*}
	where $u=e_{j_1}\otimes \cdots \otimes e_{j_k}$ is any $k$-basis atom and
	\begin{equation*}
	a(Q,\sigma,u)
	=\begin{cases}
	1&\text{if }\sigma=\id,\\
	q_{j_{m_l}j_{m_l+1}}\prod_{i=0}^{l-1}q_{j_{\varphi_{i}(m_{l-i})}j_{\varphi_{i}(m_{l-i}+1)}}&\text{if }\sigma=\sigma_{m_1}\cdots\sigma_{m_l},
	\end{cases}
	\end{equation*}
	with $\varphi_i=\sigma_{m_{l-i+1}}\cdots\sigma_{m_l}$. Notice that if $\sigma=\sigma_{m_1}\cdots\sigma_{m_l}$, the coefficient $a(Q,\sigma,u)$ is well-defined, though such representation of $\sigma$ is not unique. When all the entries of $Q$ are the same, that is, $q_{ij} \equiv q$, the operator $P^{(k)}$ reduces to 
	\begin{equation*}
	P^{(k)}(u)=\sum_{\sigma\in S_k} q^{I(\sigma)}\sigma(u).
	\end{equation*}
	Under the condition that $\sup_{i,j\in J}|q_{ij}|<1$, the operator $P^{(k)}$ is strictly positive \cite[Theorem 2.3]{BS94}.
	
	Let $\F_{Q}^{\text{finite}}$ be the subspace of finite sums of the spaces  $H^{\otimes k},k\ge 0$, where $H^{\otimes 0}=\IR\Omega$ and $\Omega$ is the vacuum vector. Then $\F_{Q}^{\text{finite}}$ is a dense subset
	of $\oplus_{k\ge 0}H^{\otimes k}$, and we define an inner product $\langle\cdot,\cdot \rangle_{Q}$ on $\F_{Q}^{\text{finite}}$ as:
	\begin{equation*}
	\langle \xi,\eta\rangle_{Q}=\delta_{kl}\langle \xi,P^{(l)}\eta\rangle_0,\text{ for }\xi\in H^{\otimes k},\eta\in H^{\otimes l},\text{ and }k,l\ge0,
	\end{equation*}
where $\langle \cdot,\cdot\rangle_0$ is the usual inner product.
	The Fock space $\F_{Q}(H)$ is the completion of  $\F_{Q}^{\text{finite}}$ with respect to the inner product $\langle\cdot,\cdot \rangle_{Q}$.  When $q_{ij}\equiv q$, the Fock space $\F_{Q}(H)$ is also denoted by $\F_{q}(H)$ for short. Notice that if we only have $\sup_{i,j\in J}|q_{ij}|\le 1$, then each $P^{(k)}$ is only positive. One should divide $\F_{Q}^{\text{finite}}$ by the kernel of $\langle\cdot,\cdot \rangle_{Q}$ before taking the completion. The definition of Fock space here is actually the same as the one in \cite{BS94} associated to the Yang--Baxter operator 
$$T:H\otimes H\to H\otimes H,~~e_i\otimes e_j\mapsto q_{ji}e_j\otimes e_{i}.$$
	See \cite[Part I]{LP99} for a detailed proof for this  when $\dim H<\infty$. 

Now we recall the mixed $q$-Gaussian algebra  $\Gamma_{Q}(H)$. For any $i\in J$, the \emph{left creation operator} $l_i$ is defined by
	\begin{equation*}
	l_i(\xi)=e_i\otimes \xi,~~\xi\in\F_{Q}(H).
	\end{equation*}
	Its adjoint with respect to $\langle \cdot,\cdot\rangle_{Q}$, the \emph{left annihilation operator} $l_i^*$, is given by
	\begin{equation*}
	l_i^*(\Omega)=0,
	\end{equation*}
	\begin{align*}
	l^*_i(e_{j_1}\otimes \cdots \otimes e_{j_k})=\sum_{m=1}^{k}&\big(\delta_{i j_m}q_{j_{m}j_{m-1}}q_{j_{m}j_{m-2}}\cdots q_{j_{m}j_{1}}\\
	&\quad e_{j_1}\otimes \cdots \otimes e_{j_{m-1}}\otimes e_{j_{m+1}}\otimes \cdots \otimes e_{j_k}\big).
	\end{align*}
	The left annihilation operators and left creation operators satisfy the deformed communication relations on $\F_{Q}(H)$:
	\begin{equation*}
	l_i^* l_j-q_{ij}l_j l_i^*=\delta_{ij}\id,~~i,j\in J.
	\end{equation*}
	The mixed $q$-Gaussian algebra $\Gamma_{Q}(H)$ is defined as the von Neumann subalgebra of $B(\F_{Q}(H))$ generated by self-adjoint operators $s_i=l_i+l_i^*,i\in J$. It is equipped with a normal faithful tracial state $\tau_Q$ given by 
	\begin{equation*}
	\tau_Q(x)=\langle x \Omega,\Omega\rangle_Q.
	\end{equation*}
	The map $\phi_{H}\colon\Gamma_{Q}(H)\to\F_{Q}(H),x\mapsto x(\Omega)$, extends to a unitary, which we still denote by $\phi_H$, from $L^2(\Gamma_{Q}(H),\tau_Q)$ to $\F_{Q}(H)$. Note that $\phi_H(s_i)=e_i$. %Using the map $\phi_{H}$, the number operator $N$ and the gradient operator $\partial$ on the Fock space $\F_{Q}(H)$ can be transferred to the mixed $q$-Gaussian algebra on $\Gamma_{Q}(H)$. More precisely, the  the number operator $N':\Gamma_{Q}(H)\to \Gamma_{Q}(H)$ and the gradient operator $\partial':\Gamma_{Q}(H)\to \Gamma_{Q}(H')$  are
	
Let $T\colon H\to H$ be a contraction. Then it induces a contraction $\F_Q(T)$ on $\F_Q(H)$ such that \cite[Lemma 1.1]{LP99} 
$$\F_Q(T)\Omega=\Omega,$$
$$\F_Q(T)(f_1\otimes \cdots \otimes f_k)=T(f_1)\otimes \cdots \otimes T(f_k),$$
for any $k$-atom $f_1\otimes \cdots \otimes f_k$ and any $k\ge 1$. Moreover, there exists a unique unital and completely positive map $\Gamma_Q(T)$ on $\Gamma_Q(H)$ such that \cite[Lemma 3.1]{LP99}
$$\Gamma_Q(T)=\phi_{H}^{-1}\F_{Q}(T)\phi_{H}.$$
%$$\Gamma_Q(T)(x)\Omega=\F_Q(T)(x\Omega),~~x\in\Gamma_Q(H).$$
Remark that $\Gamma_Q$ is a functor, that is, $\Gamma_Q(ST)=\Gamma_Q(S)\Gamma_Q(T)$ for two contractions $S,T$ on $H$. If $q_{ij}\equiv q\in[-1,1]$, then we write the functor $\Gamma_Q$ as $\Gamma_q$ for short. It interpolates between the bosonic and the fermionic functors by taking $q=+1$ and $q=-1$ respectively. When $q=0$, it becomes the free functor by Voiculescu \cite{Voi85}. For more examples, see \cite[Introduction]{LP99}.

 In particular, $T_t=T_t^Q=\F_Q(e^{-t}\id_{H})$  is a semigroup of contractions on $\F_Q(H)$.
The mixed $q$-Ornstein--Uhlenbeck semigroup is defined as $P_t=P_t^Q=\Gamma_Q(e^{-t}\id_H), t\ge 0$. It extends to a semigroup of contractions on $L^2(\Gamma_Q(H),\tau_Q)$ and is $\tau_Q$-symmetric. Note that the generator of $P_t$ is $L=\phi^{-1}_{H}N\phi_{H}$, where  $N\colon\F_{Q}^{\text{finite}}(H)\to \F_{Q}^{\text{finite}}(H)$, is the number operator defined as $k\id$ on its eigenspace $H^{\otimes k},k\ge 0$. 

Put
	\begin{equation*}
	Q'=Q\otimes \begin{pmatrix}
	1&1\\
	1&1
	\end{pmatrix},
	\end{equation*}
	and $$e=\begin{pmatrix}
	1\\
	0
	\end{pmatrix},~~
	f=\begin{pmatrix}
	0\\
	1
	\end{pmatrix}.
	$$
Then $H'\colon =H\oplus H$ can be identified with $H\otimes \mathbb{R}^2$, as a direct sum of $H\otimes \mathbb{R}e$ and $H\otimes \mathbb{R}f$. The number operator $N$ admits the following form \cite[Lemma 1.2]{LP99}: $N=\nabla^\dagger\nabla$, where $\nabla\colon\F_{Q}^{\text{finite}}(H)\to \F_{Q'}^{\text{finite}}(H')$ is the \emph{gradient operator} such that $\nabla(\Omega)=0$, and 
	\begin{equation*}
	\nabla(u)=\sum_{i=1}^{k}u\otimes v_i,
	\end{equation*}
	for $k\ge 1$, $u$ being any $k$-atom on $H$ and $v_i=e\otimes \cdots\otimes f\otimes \cdots \otimes e\in(\IR^2)^{\otimes k}$, $f$ occurring in the $i$-th factor. Remark that similar to the second quantization of any contraction $T:H\to H$, the natural embedding $\iota_H:H\to H',x\mapsto x\otimes e$ also induces a unique map $h_H\colon\Gamma_{Q}(H)\to \Gamma_{Q'}(H')$ such that \cite[Lemma 3.1]{LP99}
	\begin{equation}\label{eq:h_H}
	    h_H=\Gamma_Q(\iota_H)=\phi_{H'}^{-1}\F_Q(\iota_H)\phi_H,
	\end{equation}
	where $\F_Q(\iota_H)$ is defined as $\iota_H\otimes \cdots \otimes \iota_H$ on $H^{\otimes k}$, $k\ge 0$.
Set $\partial\colon=\phi_{H'}^{-1}\nabla \phi_{H}$. Then the generator $L$ of $P_t$ takes the form $L=\partial^\dagger\partial$ and $\partial$ is a derivation \cite[Proposition 3.2]{LP99}:
	\begin{equation*}
	\partial(xy)=\partial(x)h_{H}(y)+h_{H}(x)\partial(y),
	\end{equation*}
	for all $x,y\in \phi_H^{-1}(\F_{Q}^{\text{finite}}(H))$.
	\smallskip
	
	Now we prove that $P_t=e^{-tL}$ on $\Gamma_{Q}(H)$ satisfies $\CGE(1,\infty)$. For this let us first take a look of the semigroup $T_t=e^{-tN}$ on $\F_{Q}(H)$. By definition, it equals $e^{-kt}\id$ on its eigenspace $H^{\otimes k}$. For each $t\ge 0$, consider the map $$\vec{T}_t=e^{-t}\F_{Q'}(S_t)\colon\F_{Q'}(H')\to \F_{Q'}(H'),$$
	where $S_t$ is a contraction on $H'$ given by 
	$$S_t(x\otimes e)=e^{-t}x\otimes e,~~S_t(x\otimes f)=x\otimes f,~~x\in H.$$
	Then by definition, we have the intertwining condition
	\begin{equation}\label{eq:intertwining for Fock space}
	\nabla T_t=\vec{T}_t \nabla.
	\end{equation}
	In fact, it is obvious when acting on $\mathbb{R}\Omega$. If $u$ is a $k$-atom on $H$, $k\ge1$, then 
	\begin{equation*}
	\nabla T_t(u)=e^{-kt}\nabla (u)=e^{-kt}\sum_{i=1}^{k}u\otimes v_i,
	\end{equation*}
	and 
	\begin{equation*}
	\vec{T}_t \nabla (u)
	=\sum_{i=1}^{k}\vec{T}_t (u\otimes v_i)
	=e^{-t}\sum_{i=1}^{k}\F_{Q'}(S_t)(u\otimes v_i)
	=e^{-kt}\sum_{i=1}^{k}u\otimes v_i.
	\end{equation*}
	Remark that if one chooses $\vec{T}_t=\F_{Q'}(e^{-t}\id_{H'})$, then we can only obtain $\CGE(0,\infty)$.
	
	%is defined as $e^{-(k'+1)t}\id$ on its eigenspace $\otimes_{i=1}^{k}(H\otimes \IR g_i)$, where $g_i\in\{e,f\}$ and $k'=\sharp\{1\le i\le k:g_i=e\}$. 
	%In other words, it is the quantization of the map $j_t:\IC^{2n}\to\IC^{2n}$ for which $j_t(x\otimes e)=e^{-t}x, j_t(x\otimes f)=x$.
	%For example, we have
	%\begin{equation*}
	%\vec{P}_t(\otimes_{i=1}^{k}(u_i\otimes g_i))=e^{-(k'+1)t}\otimes_{i=1}^{k}(u_i\otimes g_i),
	%\end{equation*}
	%with $k'$ defined as above. Then the intertwining condition
	%\begin{equation*}
	%\partial \cP_t =\vec{P}_t \partial,~t\ge 0
	%\end{equation*}
	%is easy to check.
	
	%For example, if $u$ is a $k$-atom on $H$, then 
	%\begin{equation*}
	%\partial \cP_t(u)=e^{-kt}\partial (u)=e^{-kt}\sum_{i=1}^{k}u\otimes v_i,
	%\end{equation*}
	%and 
	%\begin{equation*}
	%\vec{P}_t \partial(u)=\sum_{i=1}^{k}\vec{P}_t (u\otimes v_i)=e^{-kt}\sum_{i=1}^{k}u\otimes v_i.
	%\end{equation*}
	Put $\vec{P}_t=\phi_{H'}^{-1}\vec{T}_t\phi_{H'}$. Then $\vec{P}_t$ is $\tau_{Q'}$-symmetric. Note that $P_t=\phi_H^{-1}T_t\phi_H$, thus by \eqref{eq:intertwining for Fock space} we have the intertwining condition
	\begin{equation*}
	\partial P_t 
	=\phi_{H'}^{-1}\nabla T_t\phi_{H}
	=\phi_{H'}^{-1}\vec{T}_t\nabla\phi_{H}
	=\vec{P}_t \partial,~t\ge 0.
	\end{equation*}
	Note that $S_t\circ \iota_H=e^{-t}\iota_H\circ\id_H,t\ge 0$. This, together with the definitions of $h_H$ \eqref{eq:h_H} and $\vec{P}_t$, yields
	\begin{align}
	\begin{split}\label{eq:intertwining h_H P_t}
	    \vec{P}_t h_{H}
		&=e^{-t}\phi_{H'}^{-1}\F_{Q'}(S_t)\F_Q(\iota_H)\phi_H\\
		&=e^{-t}\phi_{H'}^{-1}\F_Q(\iota_H)\F_Q(e^{-t}\id_H)\phi_H\\
		&=e^{-t}h_{H} \cP_t.
		\end{split}
	\end{align}
	By Theorem \ref{thm:intertwining}, to show that $P_t$ satisfies $\GE(1,\infty)$, it remains to check  (ii) and (iii)
	with $\vec{P}_t$ as above and the left and right action of $\Gamma_Q(H)$ on $\Gamma_{Q'}(H')$ being 
	$$L(\rho)a= h_{H}(\rho)a,~~ R(\rho)a=a h_{H}(\rho).$$
	To prove (ii) we need to show that for any $\rho\in \Gamma_Q(H)_+$ and $a\in\Gamma_{Q'}(H')$:
	\begin{equation*}
	\langle\vec{P}_t (a),L(\rho)\vec{P}_t (a)\rangle_2
	\le e^{-2t} \langle a,L(\cP_t(\rho))(a)\rangle_2,
	\end{equation*}
	where the inner product is induced by $\tau_{Q'}$. To see this, note that $\vec{P}_t$ is completely positive and $\vec{P}_t(1)=e^{-t}1$ \cite[Lemma 3.1]{LP99}. By the Kadison--Schwarz inequality and \eqref{eq:intertwining h_H P_t}, we have
	\begin{align*}
	\langle\vec{P}_t (a),L(\rho)\vec{P}_t (a)\rangle_2
	&=\tau_{Q'}\left(\vec{P}_t (a)\vec{P}_t (a)^*h_{H}(\rho)\right)\\
	&\le e^{-t}\tau_{Q'}\left(\vec{P}_t (aa^*)h_{H}(\rho)\right)\\
	&=e^{-t}\tau_{Q'}\left(aa^*\vec{P}_t h_{H}(\rho)\right)\\
	&=e^{-2t}\tau_{Q'}\left(aa^* h_{H} \cP_t(\rho)\right)\\
	&=e^{-2t} \langle a,L(\cP_t(\rho))(a)\rangle_2,
	\end{align*}
	which finishes the proof of (ii). The proof of (iii) is similar. So $P_t$ satisfies $\GE(1,\infty)$.
 Applying the same argument to $P_t\otimes \id_{\cN}$, we obtain $\CGE(1,\infty)$.
\end{example}

\begin{remark}
As mentioned in \cite[Section 4.4]{JLL20}, the previous example can also be deduced from the complete gradient estimate for the classical Ornstein--Uhlenbeck semigroup using the ultraproduct methods from \cite{JZ15b}. However, in contrast to this approach we do not need to use the Ricci curvature bound for the classical Ornstein--Uhlenbeck semigroup, but get it as a special case (with minor modifications accounting for $\abs{q}=1$ in this case).
\end{remark}

\begin{example}[Quantum Tori]
For $\theta\in [0,1)$ let $A_\theta$ be the universal $C^\ast$-algebra generated by unitaries $u=u_{\theta},v=v_{\theta}$ subject to the relation $vu=e^{2\pi i\theta}uv$. Let $\tau=\tau_{\theta}$ be the unique faithful tracial state on $A_\theta$ given by $\tau(u^m v^n)=\delta_{m,0}\delta_{n,0}$. The semigroup $(\cP_t)=(\cP^\theta_t)$ given by $\cP_t(u^m v^n)=e^{-t(m^2+n^2)}u^m v^n$ extends to a $\tau$-symmetric QMS on $L^\infty(A_\theta,\tau)$, which satisfies $\CGE(0,\infty)$. Here  $L^\infty(A_\theta,\tau)$ denotes the strong closure of $A_\theta$ in the GNS representation associated with $\tau$.  In fact, according to \cite[Section 10.6]{CS03}, $\cH=L^2(A_{\theta},\tau)\oplus L^2(A_{\theta},\tau)$ and $\partial(u^m v^n)=(\partial_1(u^m v^n),\partial_2(u^m v^n))=i(mu^m v^n,n u^m v^n)$. Clearly, $\partial_j$ commutes with $\cP_t$ for $j=1,2$, so that $\CGE(0,\infty)$ follows from Corollary \ref{cor:intertwining_dir_sum}.

In the commutative case $\theta=0$, $A_{\theta}=C(\mathbb{T}^2)$ is the C*-algebra of all continuous functions on flat $2$ torus $\mathbb{T}^2$ and the semigroup $(\cP_t)$ is the heat semigroup generated by the Laplace--Beltrami operator on the flat $2$-torus, which has vanishing Ricci curvature. Thus the constant $0$ in the gradient estimate is optimal.

In fact, for any $\theta,\theta'\in [0,1)$, the semigroup $\cP_t^\theta$ on $L^\infty(A_{\theta},\tau_\theta)$ satisfies $\CGE(K,\infty)$ if and only if the semigroup $(\cP_t^{\theta'})$ on $L^\infty(A_{\theta'},\tau_{\theta'})$ satisfies $\CGE(K,\infty)$. Thus the gradient estimate $\CGE(0,\infty)$ is optimal for any $\theta\in[0,1)$. To see this, note first that by standard approximation arguments it suffices to show $\GE(K,\infty)$ for $\rho\in (A_{\theta})_+$ and $a\in D(\cL_2^{1/2})\cap A_\theta$. By universal property of $A_{\theta+\theta'}$, there exists a $^\ast$-homomorphism $\pi\colon A_{\theta+\theta'}\to A_{\theta}\otimes A_{\theta'}$ such that 
\begin{equation*}
\pi(u_{\theta+\theta'})=u_{\theta}\otimes u_{\theta'},~~\pi(v_{\theta+\theta'})=v_{\theta}\otimes v_{\theta'}.
\end{equation*}
Clearly $\pi$ is trace preserving and satisfies
\begin{equation*}
(\cP^{\theta}_t\otimes \id_{A_{\theta'}})\circ\pi=\pi\circ \cP_t^{\theta+\theta'}.
\end{equation*}
So if $\cP^{\theta}_t$ satisfies $\CGE(K,\infty)$, then so does $\cP^{\theta+\theta'}_t$. Since $\theta$ and $\theta'$ are arbitrary, we finish the proof of the assertion. This idea of transference was used in \cite{Ric16} to give a simple proof that the completely bounded Fourier multipliers on noncommutative $L_p$-spaces associated with quantum tori $A_{\theta}$ do not depend on the parameter $\theta$. The transference technique has been used in \cite{GJL18,JLL20} to study complete logarithmic Sobolev inequality.

The same conclusion goes for $d$-dimensional quantum torus $A_{\theta}$ with $\theta$ being a $d$-by-$d$ real skew-symmetric matrix. 
\end{example}

\begin{example}[Quantum groups]\label{ex:quantum groups}
A \emph{compact quantum group} is a pair $\mathbb{G}=(A,\Delta)$ consisting of a unital C*-algebra $A$ and a unital $^\ast$-homomorphism $\Delta\colon A\to A\otimes A$ such that 
\begin{enumerate}
	\item $(\Delta\otimes\id_A)\Delta=(\id_A\otimes\Delta)\Delta$;
	\item $\{\Delta(a)(1\otimes b):a,b\in A\}$ and $\{\Delta(a)(b\otimes1):a,b\in A\}$ are linearly dense in $A\otimes A$.
\end{enumerate}
Here $A\otimes A$ is the minimal C*-algebra tensor product. The homomorphism $\Delta$ is called the \emph{comultiplication} on $A$. We denote $A=C(\IG)$. Any compact quantum group $\mathbb{G}=(A,\Delta)$ admits a unique \textit{Haar state}, i.e.\ a state $h$ on $A$ such that
\begin{equation*}
(h\otimes\id_A)\Delta(a)=h(a)1=(\id_A\otimes h)\Delta(a),~~a\in A.
\end{equation*}
Consider an element $u\in A\otimes B(H)$ with $\dim H=n$. By identifying $A\otimes B(H)$ with $M_n(A)$ we can write $u=[u_{ij}]_{i,j=1}^{n}$, where $u_{ij}\in A$. The matrix $u$ is called an \textit{n-dimensional representation} of $\mathbb{G}$ if we have
\[
\Delta(u_{ij})=\sum_{k=1}^{n}u_{ik}\otimes u_{kj},~~i,j=1,\dots,n.
\]
A representation $u$ is called \textit{unitary} if $u$ is unitary as an element in $M_n(A)$, and \textit{irreducible} if the only matrices $T\in M_n(\mathbb{C})$ such that $uT=Tu$ are multiples of identity matrix. Two representations $u,v\in M_n(A)$ are said to be \textit{equivalent} if there exists an invertible matrix $T\in M_n(\mathbb{C})$ such that $Tu=vT$. Denote by $\Irr(\mathbb{G})$ the set of equivalence classes of irreducible unitary representations of $\mathbb{G}$. For each $\alpha\in\Irr(\mathbb{G})$, denote by $u^\alpha\in A\otimes B(H_\alpha)$ a representative of the class $\alpha$, where $H_\alpha$ is the finite dimensional Hilbert space on which $u^\alpha$ acts. In the sequel we write $n_\alpha=\dim H_\alpha$.

Denote $\Pol(\mathbb{G})=\text{span} \left\{u^\alpha_{ij}:1\leq i,j\leq n_\alpha,\alpha\in\Irr(\mathbb{G})\right\}$. This is a dense subalgebra of $A$. On $\Pol(\mathbb{G})$ the Haar state $h$ is faithful. It is well-known that $(\Pol(\mathbb{G}),\Delta)$ is equipped with the Hopf*-algebra structure, that is, there exist a linear antihomormophism $S$ on $\Pol(\mathbb{G})$, called the \textit{antipode}, and a unital $^\ast$-homomorphism $\epsilon\colon\Pol(\mathbb{G})\to\mathbb{C}$, called the \textit{counit}, such that
\begin{equation*}
(\epsilon\otimes\id_{\Pol(\IG)})\Delta(a)=a=(\id_{\Pol(\IG)}\otimes\epsilon)\Delta(a),~~a\in\Pol(\mathbb{G}),
\end{equation*}
and
\begin{equation*}
m(S\otimes\id_{\Pol(\IG)})\Delta(a)=\epsilon(a)1=m(\id_{\Pol(\IG)}\otimes S)\Delta(a),~~a\in\Pol(\mathbb{G}).
\end{equation*}
Here $m$ denotes the multiplication map $m\colon\Pol(\mathbb{G})\otimes_{\text{alg}}\Pol(\mathbb{G})\to\Pol(\mathbb{G}),~~a\otimes b\mapsto ab$.
Indeed, the antipode and the counit are uniquely determined by
\begin{equation*}
S(u^\alpha_{ij})=(u^{\alpha}_{ji})^*,~~1\leq i,j\leq n_\alpha,~~\alpha\in\Irr(\mathbb{G}),
\end{equation*}
\begin{equation*}
\epsilon(u^\alpha_{ij})=\delta_{ij},~~1\leq i,j\leq n_\alpha,~~\alpha\in\Irr(\mathbb{G}).
\end{equation*}

\smallskip 

Since the Haar state $h$ is faithful on $\Pol(\IG)$, one may consider the corresponding GNS construction $(\pi_h,H_h,\xi_h)$ such that $h(x)=\langle \xi_h,\pi_h(x)\xi_h \rangle_{H_h}$ for all $x\in \Pol(\IG)$. The \emph{reduced $C^\ast$-algebra} $C_{r}(\IG)$ is the norm completion of $\pi_h(\Pol(\IG))$ in $B(H_{h})$. Then the restriction of comultiplication $\Delta$ to $\Pol(\IG)$, extends to a unital $^\ast$-homomorphism on $C_r(\IG)$, which we still denote by $\Delta$. The pair $(C_r(\IG),\Delta)$ forms a compact quantum group, and in the following we always consider this reduced version (instead of the \emph{universal} one, since the Haar state $h$ is always faithful on $C_{r}(\IG)$). Denote by $L^\infty(\IG)=C_r(\IG)''$ the von Neumann subalgebra of $B(H_h)$ generated by $C_r(\IG)$, and we can define the noncommutative $L^p$-spaces associated with $(L^\infty(\IG),h)$. In particular, we identify $L^2(\IG)$ with $H_h$. We refer to \cite{MV98} and \cite{Wor98} for more details about compact quantum groups. 

\smallskip

A compact quantum group $\IG$ is of \emph{Kac type} if the Haar state is tracial. In the following $\IG$ is always a compact quantum group of Kac type, which is the case for later examples $O_N^+$ and $S_N^+$. Given a L\'evy process $(j_t)_{t\ge 0}$ \cite[Definition 2.4]{CFK14} on $\Pol(\IG)$ one can associate it to a semigroup $P_t=(\id\otimes \phi_t)\Delta$ on $C_r(\IG)$, where $\phi_t$ is the marginal distribution of $j_t$. This $(P_t)$ is a strongly continuous semigroup of unital completely positive maps on $C_r(\IG)$ that are symmetric with respect to the Haar state $h$ \cite[Theorem 3.2]{CFK14}. Then $(\cP_t)$ extends to a $h$-symmetric QMS on $L^\infty(\IG)$.

The corresponding first-order differential calculus can be described in terms of a \emph{Schürmann triple} $((H,\pi),\eta,\varphi)$ \cite[Propositions 8.1, 8.2]{CFK14}. The tangent bimodule $\cH$ is then a submodule of $L^2(\mathbb{G})\otimes H$ with the left and right action given by $L=(\lambda_L\otimes \pi)\Delta$ and $R=\lambda_R\otimes \mathrm{id}_H$, respectively. Here $\lambda_L$ and $\lambda_R$ are the left and right action of $L^\infty(\mathbb{G})$ on $L^2(\mathbb{G})$:
$$\lambda_L(a)(b\xi_h)=ab\xi_h,~~\lambda_R(a)(b\xi_h)=ba\xi_h.$$
The derivation \cite[Proposition 8.1]{CFK14} is given on $\mathrm{Pol}(\mathbb{G})$ by $\partial=(\iota_h\otimes\eta)\Delta$, where $\iota_h\colon L^\infty(\mathbb{G})\to L^2(\mathbb{G})$ is the natural embedding:
$$\iota_h(a)=a\xi_h.$$

Note that the QMS $(P_t)$ is always \emph{right translation invariant}: $(\id\otimes P_t)\Delta=\Delta \cP_t$ for all $t\ge0$. In fact, any right translation invariant QMS must arise in this way \cite[Theorem 3.4]{CFK14}. Here
we are interested in semigroups $(\cP_t)$ that are not only right translation invariant but also \emph{left translation invariant}, or \emph{translation bi-invariant}: for all $t\ge0$
\begin{equation}\label{eq:bi-invariance}
(\cP_t\otimes \id)\Delta=\Delta \cP_t=(\id\otimes \cP_t)\Delta.
\end{equation}
In this case, let $\vec \cP_t=P_t\otimes \id_H$, and we have
\begin{equation*}
\vec \cP_t \partial=(P_t\otimes \id_H)(\iota_h\otimes \eta)\Delta=(\iota_h\otimes \eta)(P_t\otimes \id_A) \Delta=(\iota_h\otimes\eta)\Delta \cP_t=\partial \cP_t.
\end{equation*}
It is not hard to check that $\vec \cP_t$ is $\cJ$-real. We will show that it also satisfies the condition (iii) from Theorem \ref{thm:intertwining} for $K=0$.

For $\xi_1,\dots,\xi_n\in H$ and $x_1,\dots,x_n\in A$ we have
\begin{align*}
&\quad\left\langle (P_t\otimes \id_H)\sum_k x_k\otimes \xi_k,R(\rho)(P_t\otimes \id)\sum_k x_k\otimes \xi_k\right\rangle\\
&=\sum_{k,l}\langle \xi_k,\xi_l\rangle h((\cP_t x_k)^\ast (\cP_t x_l)\rho),
\end{align*}
and
\begin{equation*}
\left\langle \sum_k x_k\otimes \xi_k,R(\cP_t \rho)\sum_k x_k\otimes \xi_k\right\rangle
=\sum_{k,l}\langle \xi_k,\xi_l\rangle h(x_k^\ast x_l \cP_t\rho).
\end{equation*}
Clearly, the matrix $[\langle \xi_k,\xi_l\rangle]_{k,l}$ is positive semi-definite. By Kadison--Schwarz inequality, 
\begin{equation*}
[(\cP_t x_k)^\ast (\cP_t x_l)]_{k,l}\leq [\cP_t (x_k^\ast x_l)]_{k,l}.
\end{equation*}
Thus also $[h((\cP_t x_k)^\ast (\cP_t x_l)\rho)]_{k,l}\leq [h(x_k^\ast x_l \cP_t \rho)]_{k,l}$. Since the Hadamard product of positive semi-definite matrices is positive semi-definite, it follows that 
\begin{equation*}
[\langle \xi_k,\xi_l\rangle h((\cP_t x_k)^\ast(\cP_t x_l)\rho)]_{k,l}\leq [\langle \xi_k,\xi_l\rangle h(x_k^\ast x_l \cP_t\rho)]_{k,l}.
\end{equation*}
Hence 
\begin{equation*}
\sum_{k,l}\langle \xi_k,\xi_l\rangle h((\cP_t x_k)^\ast (\cP_t x_l)\rho)
\le\sum_{k,l}\langle \xi_k,\xi_l\rangle h(x_k^\ast x_l \cP_t\rho),
\end{equation*}
and we get the desired result. Thus $(\cP_t)$ satisfies $\GE(0,\infty)$. Applying the same argument to $(P_t\otimes \id_{\cN})$, we get $\CGE(0,\infty)$.

\smallskip

If each $\phi_t$ is \emph{central}:
\begin{equation}\label{eq:central}
(\phi_t\otimes \id)\Delta=(\id\otimes \phi_t)\Delta.
\end{equation}
then the QMS $\cP_t=(\id\otimes \phi_t)\Delta$ is translation-bi-invariant. Recall that the convolution of two functionals $\psi_1,\psi_2$ on $C(\mathbb{G})$ (or $C_r(\IG)$, $\Pol(\IG)$) is defined as $\psi_1\star \psi_2=(\psi_1\otimes \psi_2)\Delta$. The \emph{convolution semigroup of states} $\phi_t=\epsilon+\sum_{n\ge 1}\frac{t^{ n}}{n!}\psi^{\star n}$ is generated by $\psi$, called the \emph{generating functional}, where $\psi$ is hermitian, conditionally positive and vanishes on $1$ (see \cite[Section 2.5]{CFK14} for details). Then once the generating functional $\psi$ is central, the QMS $\cP_t=(\id\otimes \phi_t)\Delta=e^{tT_\psi}$ is translation-bi-invariant, and thus satisfies $\CGE(0,\infty)$, where $T_{\psi}=(\id\otimes \psi)\Delta$.

For the geometric Ricci curvature condition this result was independently proven in \cite[Lemma 4.6]{BGJ20}.
%Clearly, the same argument applies to $(\cP_t\otimes \id_\cN)$ and $(\vec \cP_t\otimes \id_\cN)$. 
\end{example}

In the next few examples we collect some specific instances of QMS on quantum groups which are translation-bi-invariant. Firstly we give some commutative examples.

\begin{example}[Compact Lie groups]
For any compact group $G$, $(C(G),\Delta)$ forms a compact quantum group, where $C(G)$ is the C*-algebra of all continuous functions on $G$ and the comultiplication $\Delta\colon C(G)\to C(G)\otimes C(G)\cong C(G\times G)$ is given by $\Delta f(s,t)=f(st)$. The Haar state $h$ is nothing but $\int \cdot\, d\mu$, with $\mu$ being the Haar (probability) measure. Consider the QMS $(P_t)$ on $C(G)$: $P_t(f)(s)=\int_{G}f(r)K_t(r,s)d\mu(r)$. Then $(P_t)$ is translation bi-invariant if and only if the kernel $K_t$ is bi-invariant under $G$: $K_t(gr,gs)=K_t(r,s)=K_t(rg,sg)$ for all $g,r,s\in G$, or equivalently, $(P_t)$ is a convolution semigroup with the kernel $\tilde{K}_t(s)=K(e,s)$ being conjugate-invariant: $\tilde{K}(s)=\tilde{K}(gsg^{-1})$ for all $g,s\in G$.

Let $G$ be a compact Lie group with a bi-invariant Riemann metric $g$. If $(\cP_t)$ is the heat semigroup generated by the Laplace--Beltrami operator, then a direct computation shows that the bi-invariance of the metric implies the translation-bi-invariance of $(\cP_t)$. Thus we recover the well-known fact from Riemannian geometry that the Ricci curvature of a compact Lie group with bi-invariant metric is always nonnegative (see e.g. \cite[Section 7]{Mil76}). 
\end{example}

Secondly, we give co-commutative examples. By saying co-commutative we mean $\Delta=\Pi\circ\Delta$, where $\Pi$ is the tensor flip, i.e., $\Pi(a\otimes b)=b\otimes a$. 
\begin{example}[Group von Neumann algebras]\label{ex:group_alg}
Let $G$ be a countable discrete group with unit $e$, $C_r^\ast(G)$ the reduced $C^\ast$-algebra generated by the left regular representation $\lambda$ of $G$ on $\ell^2(G)$ and $L(G)$ the group von Neumann algebra $L(G)=C_r^\ast(G)^{\prime\prime}\subset B(\ell^2(G))$. Then $\IG=(C_r^\ast(G),\Delta)$ is a quantum group with comultiplication given by $\Delta(\lambda_g)=\lambda_g\otimes \lambda_g$. The Haar state on $\IG$ is given by $\tau(x)=\langle x\delta_e,\delta_e\rangle$, which is tracial and faithful. Here and in what follows, $\delta_g$ always denotes the function on $G$ that takes value 1 at $g$ and vanishes elsewhere.

A function $\psi\colon G\to [0,\infty)$ is a \emph{conditionally negative definite} (cnd) length function if $\psi(e)=0$, $\psi(g^{-1})=\psi(g)$ and 
\begin{equation*}
\sum_{g,h\in G}\overline{f(g)}f(h)\psi(g^{-1}h)\leq 0
\end{equation*}
for every $f\in G\to \IC$ with finite support such that $\sum_{g\in G} f(g)=0$.

By Schoenberg's Theorem (see for example \cite[Theorem D.11]{BO08}), to every cnd function one can associate a $\tau$-symmetric QMS on $L(G)$ given by
\begin{equation*}
\cP_t \lambda_g=e^{-t\psi(g)}\lambda_g.
\end{equation*}
It is easy to check that $(\cP_t)$ satisfies the translation-bi-invariant condition (\ref{eq:bi-invariance}). Thus it satisfies $\CGE(0,\infty)$.

Now we give some genuine quantum group examples.

\begin{example}[Free orthogonal quantum group $O^+_N$ \cite{Wan95}]\label{ex:free orthogonal quantum group}
	 Let $N\ge2$. The free orthogonal quantum group $O^+_N$ consists of a pair $(C_u(O^+_N),\Delta)$, where $C_u(O^+_N)$ is the universal C*-algebra generated by $N^2$ self-adjoint operators $u_{ij},1\le i,j\le N$, such that $U=[u_{ij}]_{1\le i,j\le N}\in M_N(\mathbb{C})\otimes C_u(O^+_N)$ is unitary, that is,
		\begin{equation*}
		\sum_{k=1}^{N}u_{ik}u_{jk}=\delta_{ij}=\sum_{k=1}^{N}u_{ki}u_{kj},~~1\le i,j\le N,
		\end{equation*}
	and the comultiplication $\Delta$ is given by 
	\begin{equation*}
	\Delta(u_{ij})=\sum_{k=1}^{N}u_{ik}\otimes u_{kj},~~1\le i,j\le N.
	\end{equation*}
	The equivalent classes of irreducible unitary representations of $O_N^+$ can be indexed by $\mathbb{N}$, with $u^{(0)}=1$ the trivial representation and $u^{(1)}=U$ the fundamental representation. 
	%The dense *-Hopf algebra $\Pol(O_N^+)$ is the *-algebra generated by all the $u_{ij},1\le i,j\le N$. 
	%with $u^{(0)}=1$ the trivial representation and $u^{(1)}=V$ the fundamental representation. The dense *-Hopf algebra $\Pol(O_N^+)$ is the *-algebra generated by all the $v_{ij},1\le i,j\le N$. 
	By \cite[Corollary 10.3]{CFK14}, the central generating functionals $\psi$ on $\Pol(O^+_N)$ are given by
	\begin{equation*}
	\psi(u^{(s)}_{ij})=\frac{\delta_{ij}}{U_s(N)}\left(-bU_s'(N)+\int_{-N}^{N}\frac{U_s(x)-U_s(N)}{N-x}\nu(dx)\right),
	\end{equation*}
	for $s\in \Irr(O_N^+)=\mathbb{N},1\le i,j\le n_s$, where $U_s$ denotes the $s$-th Chebyshev polynomial of second kind, $b\ge 0$, and $\nu$ is a finite measure on $[-N,N]$ with $\nu(\{N\})=0$. Then given $(b,\nu)$, the central functional $\psi$ defined as above induces a QMS $\cP_t^\psi=e^{t T_\psi}$ satisfying \eqref{eq:bi-invariance}, where $T_\psi=(\id\otimes \psi)\Delta$. Hence it satisfies $\CGE(0,\infty$).
	
	%Note that when $N>2$, $O_N^+$ is not co-amenable, so we should first consider the semigroup over the reduced C*-algebra $C_r(O_N^+)$.
\end{example}	

\begin{example}[Free permutation quantum group $S^+_N$ \cite{Wan98}]\label{ex:free permutation quantum group}
	 Let $N\ge2$. The free permutation quantum group $S^+_N$ consists of a pair $(C_u(S^+_N),\Delta)$, where $C_u(O^+_N)$ is the universal C*-algebra generated by $N^2$ self-adjoint operators $p_{ij},1\le i,j\le N$, such that 
	\begin{equation*}
	p_{ij}^2=p_{ij}=p_{ij}^*,~~\sum_{k=1}^{N}p_{ik}=1=\sum_{k=1}^{N}p_{kj},~~1\le i,j\le N,
	\end{equation*}
	and the comultiplication $\Delta$ is given by 
	\begin{equation*}
	\Delta(p_{ij})=\sum_{k=1}^{N}p_{ik}\otimes p_{kj},~~1\le i,j\le N.
	\end{equation*}
	The equivalent classes of irreducible unitary representations of $S_N^+$ can be indexed by $\mathbb{N}$.
	By \cite[Theorem 10.10]{FKS16}, the central generating functionals $\psi$ on $\Pol(S^+_N)$ are given by
	\begin{equation*}
	\psi(u^{(s)}_{ij})=\frac{\delta_{ij}}{U_{2s}(\sqrt{N})}\left(-b\frac{U_{2s}'(\sqrt{N})}{2\sqrt{N}}+\int_{0}^{N}\frac{U_{2s}(\sqrt{x})-U_{2s}(\sqrt{N})}{N-x}\nu(dx)\right),
	\end{equation*}
	for $s\in \Irr(S_N^+)=\mathbb{N},1\le i,j\le n_s$, where $U_s$ denotes the $s$-th Chebyshev polynomial of second kind, $b> 0$, and $\nu$ is a finite measure on $[0,N]$. Similarly, given $(b,\nu)$, the central functional $\psi$ defined as above induces a QMS $\cP_t^\psi=e^{t T_\psi}$ satisfying \eqref{eq:bi-invariance}, where $T_\psi=(\id\otimes \psi)\Delta$. Hence it satisfies $\CGE(0,\infty$).
\end{example}
\end{example}

\begin{remark}
Although many interesting functional inequalities like the Poincaré and the modified logarithmic Sobolev inequality only follow directly from $\GE(K,\infty)$ for $K>0$, the gradient estimate with constant $K\leq 0$ can still be helpful in conjunction with additional assumptions to prove such functional inequalities (see \cite{DR20,BGJ20}).
\end{remark}

\section{Stability under tensor products and free products}\label{sec:stability}
In this section we prove that the complete gradient estimate $\CGE(K,\infty)$ is stable under taking tensor products and free products of quantum Markov semigroups. We refer to \cite{VDN92} and \cite{BD01} for more
information on free products of von Neumann algebras and to \cite{Boc91} for free products of completely positive maps.
\begin{theorem}\label{thm:tensor_product}
Let $(\cM_j,\tau_j)$, $j\in \{1,\dots,n\}$, be tracial von Neumann algebras and $(\cP_t^j)$ a  $\tau_j$-symmetric $\Gamma$-regular QMS on $\cM_j$. If for every $j\in\{1,\dots,n\}$ the QMS $(\cP_t^j)$ satisfies $\CGE(K,\infty)$, then $\bigotimes_j \cP_t^j$ satisfies $\CGE(K,\infty)$.
\end{theorem}
\begin{proof}
Let $\cH_j$ and $\partial_j$ denote the tangent bimodule and derivation for $(\cP_t^j)$ and let 
\begin{align*}
\bar \cH_j&=\bigotimes_{k=1}^{j-1}L^2(\cM_k,\tau_k)\otimes \cH_j \otimes \bigotimes_{k=j+1}^n L^2(\cM_k,\tau_k),\\
\bar\partial_j&=\bigotimes_{k=1}^{j-1}\mathrm{id}_{\cM_k}\otimes \partial_j \otimes \bigotimes_{k=j+1}^n \mathrm{id}_{\cM_k}.
\end{align*}
The tangent module $\cH$ for $\cP_t=\bigotimes_j \cP_t^j$ is a submodule of $\cH=\bigoplus_j \bar \cH_j$ with the natural left and right action and derivation $\partial=(\bar\partial_1,\dots,\bar\partial_n)$.

For $j\in\{1,\dots,n\}$, put 
\begin{equation*}
\tilde \cP_t^j=\bigotimes_{k=1}^{j-1}\cP_t^k\otimes \mathrm{id}_{\cM_j}\otimes\bigotimes_{k=j+1}^n \cP_t^k
\end{equation*}
 and 
\begin{equation*}
\bar \cP_t^j=\bigotimes_{k=1}^{j-1}\id_{\cM_k} \otimes \cP_t^j\otimes\bigotimes_{k=j+1}^n \id_{\cM_k}
\end{equation*}
on $\bigotimes_k \cM_k$, so that $\cP_t=\bar \cP_t^j \tilde \cP_t^j=\tilde \cP_t^j \bar \cP_t^j$. Then
\begin{align*}
\norm{\partial \cP_t a}_\rho^2&=\sum_{j=1}^n \norm{\bar \partial_j \cP_t a}_\rho^2\\
&=\sum_{j=1}^n \norm{\bar \partial_j \bar \cP_t^j \tilde \cP_t^j a}_\rho^2\\
&\leq \sum_{j=1}^n e^{-2K t}\norm{\bar \partial_j \tilde \cP_t^j a}_{\bar \cP_t^j\rho}^2
\end{align*}
by $\CGE(K,\infty)$ for  $(\cP_t^j)$.

Let
\begin{equation*}
Q_t^j=\bigotimes_{k=1}^{j-1}\cP_t^k \otimes \mathrm{id}_{\cH_j}\otimes \bigotimes_{k=j+1}^n \cP_t^k
\end{equation*}
on $\bar \cH_j$. Then $\bar \partial_j \tilde \cP_t^j =Q_t^j \bar\partial_j$, and conditions (ii), (iii) in Theorem \ref{thm:intertwining} follow from the Kadison--Schwarz inequality (compare with Example \ref{ex:quantum groups}). Taking into account Remark \ref{rmk:diff_calc_ind}, we get
\begin{align*}
\norm{\bar\partial_j \tilde \cP_t^j a}_\rho^2\leq \norm{\bar\partial_j a}_{\tilde \cP_t^j \rho}^2.
\end{align*}
Together with the previous estimate, we obtain
\begin{equation*}
\norm{\partial \cP_t a}_\rho^2
\leq \sum_{j=1}^n e^{-2K t}\norm{\bar \partial_j \tilde \cP_t^j a}_{\bar \cP_t^j\rho}^2
\leq\sum_{j=1}^n e^{-2Kt}\norm{\bar\partial_j a}_{\cP_t \rho}^2
= e^{-2Kt}\norm{\partial a}_{\cP_t \rho}^2.
\end{equation*}
So $(\cP_t)$ satisfies $\GE(K,\infty)$. The same argument can be applied to $(\cP_t\otimes \id_\cN)$, so that we obtain $\CGE(K,\infty)$.
\end{proof}

\begin{theorem}\label{thm:free_product}
For $j\in\{1,\dots,n\}$ let $(\cM_j,\tau_j)$ be a tracial von Neumann algebras and $(\cP_t^j)$ a tracially symmetric $\Gamma$-regular QMS on $\cM_j$. If for every $j\in\{1,\dots,n\}$ the QMS $(\cP_t^j)$ satisfies $\CGE(K,\infty)$, then $\ast_j \cP_t^j$ satisfies $\CGE(K,\infty)$.
\end{theorem}
\begin{proof}
Let $\cM=\ast_j \cM_j$, $\tau=\ast_j \tau_j$ and $\cP_t=\ast_j \cP_t^j$. Recall that $L^2(\cM,\tau)$ is canonically identified with
\begin{align*}
\ast_j L^2(\cM_j,\tau_j)=\IC 1\oplus \bigoplus_{n\geq 1}\bigoplus_{j_1\neq\dots\neq j_n}\bigotimes_{l=1}^n L^2_0(\cM_{j_l},\tau_{j_l}),
\end{align*}
where $L^2_0$ denotes the orthogonal complement of $\IC 1$ in $L^2$.

Then $\cH$ can be identified with a submodule of
\begin{equation*}
\bigoplus_{n\geq 1}\bigoplus_{j_1\neq\dots\neq j_n}\bigoplus_{k=1}^n\left(\bigotimes_{l=1}^{k-1} L^2(\cM_{j_l},\tau_{j_l})\otimes \cH_{j_k}\otimes \bigotimes_{l=k+1}^n L^2(\cM_{j_l},\tau_{j_l})\right)
\end{equation*}
with the natural left and right action on each direct summand and $\partial$ acts as $0$ on $\IC 1$ and as 
\begin{align*}
\partial(a_1\otimes\dots\otimes a_n)=(\partial_{j_1}(a_1)\otimes a_2\dots \otimes a_n,\dots,a_1\otimes a_2\otimes \dots\otimes \partial_{j_n}(a_n))
\end{align*}
on the direct summand $\bigotimes_{j_1\neq\dots\neq j_n}L^2(\cM_{j_l},\tau_{j_l})$. Since $\partial$ and $(\cP_t)$ restrict nicely to the direct summand of $L^2(\cM,\tau)$, the rest of the proof is similar to the one of Theorem \ref{thm:tensor_product}.
\end{proof}

\begin{remark}
The same argument applies to free products with amalgamation if the common subalgebra over which one amalgates is contained in the fixed-point algebra of $(\cP_t^j)$ for all $j\in\{1,\dots,n\}$ (compare with the results from \cite[Section 6.2]{JZ15a} for the $\Gamma_2$ condition).
\end{remark}

\section{Quantum Markov semigroups generated by commuting projections}\label{sec:com_proj}
In this section we move beyond applications of the intertwining result Theorem \ref{thm:intertwining} and obtain complete gradient estimate for quantum Markov semigroups whose generators take special Lindblad forms.

\begin{theorem}\label{thm:com_proj}
Let $p_1,\dots,p_n\in \cM$ be commuting projections. The QMS $(\cP_t)$ generated by
\begin{align*}
\cL\colon \cM\to\cM,\,\cL x=\sum_{j=1}^n p_j x+x p_j-2p_j x p_j
\end{align*}
is $\Gamma$-regular and satisfies $\CGE(1,\infty)$.
\end{theorem}

\begin{proof}
For $1\leq j\leq n$ consider the operator $\cL_j\colon\cM\to \cM$ defined by
\begin{equation*}
\cL_j x=p_j x+xp_j-2p_j xp_j=x-p_j xp_j-(1-p_j)x(1-p_j).
\end{equation*}
In particular, $\cL_j$ is of the form $\cL_j=I-\Phi_j$ with $I=\id_{\cM}$ and the conditional expectation $\Phi_j(x)=p_j x p_j+(1-p_j)x(1-p_j)$. Thus the QMS $(\cP_t^j)$ generated by $\cL_j$ is given by 
\begin{align*}
\cP_t^j x=x+(e^{-t}-1)\cL_j x=e^{-t}x+(1-e^{-t})\Phi_j(x).
%e^{-t}x+(1-e^{-t})(p_j x p_j+(1-p_j)x(1-p_j)).
\end{align*}
A first-order differential calculus for $(\cP_t)$ is given by $\cH=\bigoplus_{j=1}^n L^2(\cM,\tau)$ as bimodules, $L=(L_j)_j,R=(R_j)_j$ with $L_j$ and $R_j$ being the usual left and right multiplications of $\cM$ on $L^2(\cM,\tau)$ respectively, and $\partial=(\partial_j)$, where $\partial_j x=[p_j,x]$. Thus $(\cP_t)$ is $\Gamma$-regular.

Moreover, $\partial_j \cP_t^j x=e^{-t}\partial_j x$ and consequently
\begin{equation}\label{eq:partial_P_t}
\norm{\partial_j \cP_t^j x}_\rho^2=e^{-2t}\norm{\partial_j x}_\rho^2.
\end{equation}
On the other hand, by the concavity of operator means \cite[Theorem 3.5]{KA80} we have
\begin{equation}\label{eq:conc_decom_QMS}
\widehat{\cP_t^j \rho}\geq e^{-t}\hat \rho+(1-e^{-t})\widehat{ \Phi_j(\rho)}.
%(p_j \rho p_j+(1-p_j)\rho(1-p_j))^\wedge.
\end{equation}
Since
\begin{align*}
&\quad\;\cL_j ((\partial_j x)^\ast (\partial_j x))\\
&=p_j x^\ast x p_j+p_j x^\ast p_j x-p_j x^\ast p_j x-p_j x^\ast p_j x p_j\\
&\quad+p_j x^\ast x p_j +x^\ast p_j x p_j-p_j x^\ast p_j x p_j-x^\ast p_j x p_j\\
&\quad -2 p_j x^\ast x p_j-2 p_j x^\ast p_j x p_j+2p_j x^\ast p_j v p_j+2 p_j x^\ast p_j x p_j\\
&=0,
\end{align*}
we have 
$$\Phi_j((\partial_j x)^\ast (\partial_j x))
=(I-\cL_j)\left((\partial_j x)^\ast (\partial_j x)\right)
=(\partial_j x)^\ast (\partial_j x).$$
%$\cP_t^j ((\partial_j x)^\ast (\partial_j x))=(\partial_j x)^\ast (\partial_j x)$ and therefore
%\begin{align*}
%&\quad\;\Phi_j((\partial_j x)^\ast (\partial_j x))\\
%&=p_j(\partial_j x)^\ast (\partial_j x)p_j+(1-p_j)(\partial_j x)^\ast (\partial_j x)(1-p_j)\\
%&=(1-e^{-t})^{-1}(\cP_t^j ((\partial_j x)^\ast (\partial_j x))-e^{-t}(\partial_j x)^\ast (\partial_j x))\\
%&=(\partial_j x)^\ast (\partial_j x).
%\end{align*}
Recall that $L_j$ and $R_j$ are respectively the usual left and right multiplications of $\cM$ on $L^2(\cM,\tau)$ and denote by $E_j$ the projection onto $\overline{\ran \partial_j}$ in $L^2(\cM,\tau)$. It follows that
\begin{align*}
\langle R_j(\Phi_j(\rho))(\partial_j x),\partial_j x\rangle_2
&=\tau(\Phi_j(\rho)(\partial_j x)^\ast (\partial_j x))\\
&=\tau(\rho \Phi_j((\partial_j x)^\ast (\partial_j x)))\\
&=\tau(\rho (\partial_j x)^\ast (\partial_j x))\\
&=\langle R_j(\rho)\partial_j x,\partial_j x\rangle_2.
\end{align*}
Hence $E_j R_j(\Phi_j(\rho)) E_j=E_j R_j(\rho)E_j$. The analogous identity for the left multiplication follows similarly.

Note that both the left and right multiplication by $\Phi_j(x)=p_j x p_j+(1-p_j)x(1-p_j)$ leave $\overline{\ran \partial_j}$ invariant. In fact, for any $x,y\in \cM$ one has 
\begin{align*}
\Phi_j(x)\partial_j(y)
&=p_j(p_jx p_j y)- (p_jx p_j y )p_j\\
&\quad\;+p_j((1-p_j)x(1-p_j)y)-((1-p_j)x(1-p_j)y)p_j\\
&=\partial_j(p_j x p_j y)+\partial_j ((1-p_j)x(1-p_j)y),
\end{align*}
and a similar equation holds for the right multiplication.

Therefore we have
\begin{align*}
E_j L_j(\Phi_j(\rho))E_j&\le L_j(\Phi_j(\rho)),\\
E_j R_j(\Phi_j(\rho))E_j&\le R_j(\Phi_j(\rho)).
\end{align*}
This, together with the conditions (a) and (b) in the definition of operator means, implies
\begin{align*}
E_j\widehat{\Phi_j(\rho)} E_j
&\ge E_j\Lambda(E_j L_j(\Phi_j(\rho))E_j,E_j R_j(\Phi_j(\rho)) E_j)E_j\\
&=E_j\Lambda(E_j L_j(\rho)E_j,E_j R_j(\rho)E_j)E_j\\
&\geq E_j \hat \rho E_j.
\end{align*}
In other words,
\begin{equation*}
\langle \widehat{\Phi_j(\rho)} \partial_j x,\partial_j x\rangle_2\geq \langle \hat \rho \partial_j x,\partial_j x\rangle_2.
\end{equation*}
Together with (\ref{eq:conc_decom_QMS}) we conclude
\begin{equation*}
\norm{\partial_j x}_{\cP_t^j \rho}^2\geq e^{-t}\norm{\partial_j x}_\rho^2+(1-e^{-t})\norm{\partial_j x}_\rho^2=\norm{\partial_j x}_\rho^2.
\end{equation*}
In view of \eqref{eq:partial_P_t}, we have proved 
\begin{equation}\label{eq:estimate_P_t^j}
    \norm{\partial_j \cP_t^j x}_\rho^2\le e^{-2t}\norm{\partial_j x}_{\cP_t^j \rho}^2.
\end{equation}
Now let us come back to our original semigroup $(\cP_t)$. Let 
\begin{equation*}
Q_t^j=\prod_{k\neq j} \cP_t^k.
\end{equation*}
Since the $p_j$'s commute, so do the generators $\cL_j$'s and the semigroups $\cP_t^j$'s. This means that the order in the definition of $Q_t^j$ does not matter and $\cP_t =\cP_t^j Q_t^j$ for all $j\in\{1,\dots,n\}$. From the intertwining technique and Remark \ref{rmk:diff_calc_ind} we deduce
\begin{equation*}
\norm{\partial_j Q_t^j x}_\rho^2\leq \norm{\partial_j x}_{Q_t^j \rho}^2.
\end{equation*}
Combined with the estimate \eqref{eq:estimate_P_t^j} for $(\cP_t^j)$, we obtain
\begin{equation*}
\norm{\partial \cP_t x}_\rho^2=\sum_{j=1}^n \norm{\partial_j \cP_t^j Q_t^j x}_\rho^2\leq e^{-2t}\sum_{j=1}^n \norm{\partial_j Q_t^j x}_{\cP_t ^j \rho}^2\leq e^{-2t} \norm{\partial x}_{\cP_t \rho}^2.
\end{equation*}
So $(P_t)$ satisfies $\GE(1,\infty)$. To prove $\CGE(1,\infty)$, it suffices to note that the generator of $(\cP_t\otimes \id_\cN)$ is given by
\begin{equation*}
(\cL\otimes \id_\cN)x=\sum_{j=1}^n (p_j\otimes 1) x+x(p_j\otimes 1)-2(p_j\otimes 1) x(p_j\otimes 1)
\end{equation*}
and the elements $(p_j\otimes 1)$ are again commuting projections.
\end{proof}

\begin{remark}
Since $\cL_j^2=\cL_j$, the spectrum of $\cL_j$ is contained in $\{0,1\}$ with equality unless $\cL_j=0$. Thus the gradient estimate for the individual semigroups $(\cP_t^j)$ is optimal (unless $v_j=0$). It should also be noted that it is better than the gradient estimate one would get from Example \ref{ex:cond_exp}.
\end{remark}

\begin{remark}
Inspection of the proof shows that the same result holds if the generator of $(\cP_t)$ is of the form $\cL=\frac 1 2\sum_{j=1}^n (x-u_j xu_j)$ with commuting self-adjoint unitaries $u_j$.
\end{remark}

\begin{example}
Let $X=\{0,1\}^n$ and $\epsilon_j\colon X\to X$ the map that swaps the $j$-th coordinate and leaves the other coordinates fixed. Let $v_j=\sum_x \ket{\epsilon_j(x)}\bra{x}\in B(\ell^2(X))$. By the previous remark, the QMS on $B(\ell^2(X))$ with generator 
\begin{align*}
\cL \colon B(\ell^2(X))\to B(\ell^2(X)),\,\cL A=\frac 1 2\sum_{j=1}^n (A-v_j A v_j)
\end{align*}
satisfies $\CGE(1,\infty)$. The restriction of this semigroup to the diagonal algebra is (up to rescaling of the time parameter, depending on the normalization) the Markov semigroup associated with the simple random walk on the discrete hypercube (see \cite[Example 5.7]{EM12}).
\end{example}

To apply the theorem above to group von Neumann algebras, we will use the following Lindblad form for QMS generated by cnd length functions. Recall that for a countable discrete group $G$, a $1$-cocycle is a triple $(H,\pi,b)$, where $H$ is a real Hilbert space, $\pi\colon G\to O(H)$ is an orthogonal representation, and $b\colon G\to H$ satisfies the cocycle law: $b(gh)=b(g)+\pi(g)b(h),g,h\in G.$ To any cnd function $\psi$ on a countable discrete group $G$, one can associate with a $1$-cocycle  $(H,\pi,b)$ such that $\psi(gh^{-1})=\|b(g)-b(h)\|^2,g,h\in G$. See \cite[Appendix D]{BO08} for more information.

\begin{lemma}
Let $G$ be a countable discrete group and $\psi\colon G\to [0,\infty)$ a cnd length function. Then $\cL\colon\lambda_g\mapsto \psi(g)\lambda_g$ generates a QMS on the group von Neumann algebra of $G$. Assume that the associated $1$-cocycle $b\colon G\to H$ takes values in a finite-dimensional real Hilbert space $H$ with an orthonormal basis $(e_1,\dots,e_n)$. Then the generator $\cL$ is of the form
\begin{align*}
\cL x=\sum_{j=1}^n v_j^2 x+x v_j^2 -2v_j x v_j,
\end{align*}
where $v_j$ is a linear operator on $\ell^2(G)$ given by $v_j \delta_g=\langle b(g),e_j\rangle \delta_g$.
\end{lemma}
\begin{proof}
By definition we have
\begin{align*}
v_j^2 \lambda_g(\delta_h)&=v_j^2 (\delta_{gh})=\langle b(gh),e_j\rangle v_j(\delta_{gh})=\langle b(gh),e_j\rangle^2\delta_{gh},\\
\lambda_g v_j^2(\delta_h)&=\langle b(h),e_j\rangle\lambda_g v_j(\delta_h)=\langle b(h),e_j\rangle^2\lambda_g (\delta_h)=\langle b(h),e_j\rangle^2\delta_{gh},\\
v_j\lambda_g v_j(\delta_h)&=\langle b(h),e_j\rangle v_j\lambda_g(\delta_h)=\langle b(h),e_j\rangle v_j(\delta_{gh})=\langle b(h),e_j\rangle \langle b(gh),e_j\rangle\delta_{gh}.
\end{align*}
Thus 
\begin{equation*}
\begin{split}
&\sum_{j}\left(v_j^2\lambda_g+\lambda_g v_j^2-2v_j\lambda_g v_j\right)(\delta_h)\\
=&\sum_{j}\left(\langle b(gh),e_j\rangle^2+\langle b(h),e_j\rangle^2-2\langle b(h),e_j\rangle \langle b(gh),e_j\rangle\right)\delta_{gh}\\
=&\sum_{j}\langle b(gh)-b(h),e_j\rangle^2 \delta_{gh}\\
=&\|b(gh)-b(h)\|^2\delta_{gh}.
\end{split}
\end{equation*}
This is nothing but $\cL(\lambda_g)(\delta_h)=\psi(g)\lambda_g(\delta_h)=\psi(g)\delta_{gh}$.
\end{proof}

\begin{remark}\label{rmk:extension_group_vna}
The elements $v_j$ are not contained in the group von Neumann algebra $L(G)$ so that Theorem \ref{thm:com_proj} is not directly applicable (even if the $v_j$ are projections). However, if $G$ is finite, then the operator
\begin{equation*}
\cL\colon B(\ell^2(G))\to B(\ell^2(G)),\,\cL x=\sum_{j=1}^n v_j^2 x+x v_j^2 -2v_j x v_j,
\end{equation*}
generates a tracially symmetric QMS on $B(\ell^2(G))$ and we can apply Theorem \ref{thm:com_proj} to that semigroup instead. It is an interesting open question how to treat infinite groups for which the generator has such a Lindblad form.
\end{remark}

\begin{example}
The cyclic group $\mathbb{Z}_n=\{0,1,\dots,n-1\}$; see \cite[Example 5.9]{JZ15a}: Its group (von Neumann) algebra is spanned by $\lambda_k,0\le k\le n-1$. One can embed $\mathbb{Z}_n$ to $\mathbb{Z}_{2n}$, so let us assume that $n$ is even. The word length of $k\in\mathbb{Z}_n$ is given by $\psi(k)=\min\{k,n-k\}$. Define $b\colon\mathbb{Z}_n\to \mathbb{R}^{\frac{n}{2}}$ via\begin{equation*}
b(k)=\begin{cases}
0,&k=0,\\
\sum_{j=1}^{k}e_j,&1\le k\le \frac{n}{2},\\
\sum_{j=k-\frac{n}{2}+1}^{\frac{n}{2}}e_j,&\frac{n}{2}+1\le k\le n-1,
\end{cases}
\end{equation*}
where $(e_j)_{1\le j\le \frac{n}{2}}$ is an orthonormal basis of $\mathbb{R}^{\frac{n}{2}}$.

Then the linear operators $v_j\colon\ell^2(\mathbb{Z}_n)\to \ell^2(\mathbb{Z}_n)$ given by
\begin{equation*}
v_j(\delta_k)=\langle b(k),e_j\rangle \delta_k,~~1\le j\le \frac{n}{2}
\end{equation*}
are commuting projections. Thus the QMS associated with $\psi(g)=\norm{b(g)}^2$ satisfies $\CGE(1,\infty)$.
\end{example}

\begin{example}\label{ex:symmetric_group}
The symmetric group $S_n$: Let $\psi$ be the length function induced by the (non-normalized) Hamming metric, that is, $\psi(\sigma)=\#\{j : \sigma(j)\neq j\}$. Let $A_\sigma\in M_n(\IR)$ be the permutation matrix associated with $\sigma$, i.e., $A_\sigma \delta_j =\delta_{\sigma(j)}$. Then the associated $1$-cocycle is given by $H=L^2(M_n(\IR),\frac 1 2 \mathrm{tr})$, $b(\sigma)=A_\sigma-1$, $\pi(\sigma)=A_\sigma$.

The matrices $E_{jk}=\sqrt{2}\ket{j}\bra{k}$ for $j\neq k$ and $E_{jj}=-\sqrt{2}\ket{j}\bra{j}$ form an orthonormal basis of $H$. Define $v_{jk}\in B(\ell^2(S_n))$ by $v_{jk}\delta_\sigma=\sqrt{2}\langle b(\sigma),E_{jk}\rangle \delta_{\sigma}$. Then $v_{jk}$ is a projection. Moreover,
\begin{align*}
\cL x=\frac 1 2\sum_{j,k}v_{jk}^2x+x v_{jk}^2-2v_{jk}x v_{jk}.
\end{align*}
Thus the associated QMS satisfies $\CGE(1/2,\infty)$.
\end{example}

To extend the last example to the infinite symmetric group $S_\infty$, we need the following approximation result.

\begin{lemma}
Let $(\cM_n)$ be an ascending sequence of von Neumann subalgebras such that $\bigcup_n \cM_n$ is $\sigma$-weakly dense in $\cM$. Further let $(\cP_t)$ be a $\Gamma$-regular QMS on $\cM$ and assume that $\cP_t(\cM_n)\subset \cM_n$. Let $(\cP_t^n)$ denote the restriction of $(\cP_t)$ to $\cM_n$. If $(\cP_t^n)$ satisfies $\mathrm{GE}(K,\infty)$ for all $n\in\mathbb{N}$, then $(\cP_t)$ also satisfies $\mathrm{GE}(K,\infty)$. The same is true for $\CGE$.
\end{lemma}
\begin{proof}
It is not hard to see that $\bigcup_n \cM_n$ is dense in $L^2(\cM,\tau)$. Since $\cP_t (\cM_n)\subset \cM_n$ and $\cP_t$ maps into the domain of its $L^2$ generator $\cL_2$, the space $V=D(\cL_2^{1/2})\cap \left(\bigcup_n \cM_n\right)$ is also dense in $L^2(\cM,\tau)$ and invariant under $(\cP_t)$. Using a standard result in semigroup theory, this implies that $V$ is a form core for $\cL$. Thus it suffices to prove
\begin{equation*}
\norm{\partial \cP_t a}_\rho^2\leq e^{-2Kt}\norm{\partial a}_{\cP_t \rho}^2
\end{equation*}
for $a\in V$ and $\rho\in \cM_+$. Moreover, by Kaplansky's density theorem and the strong continuity of functional calculus, checking it for $\rho \in (\bigcup_n \cM_n)_+$ is enough. But for $a\in D(\cL^{1/2}_2)\cap \cM_n$ and $\rho\in(\cM_n)_+$, this is simply the gradient estimate for $(\cP_t^n)$, which holds by assumption.

The argument for $\CGE$ is similar.
\end{proof}

\begin{corollary}
If $G$ is the ascending union of subgroups $G_n$ and $\psi$ is a cnd length function on $G$ such that for every $n$ the QMS associated with $\psi|_{G_n}$ satisfies $\GE(K,\infty)$, then the QMS associated with $\psi$ satisfies $\GE(K,\infty)$. The same is true for $\CGE$.
\end{corollary}

\begin{example}[Infinite symmetric group]
Let $S_\infty$ be the group of permutations of $\mathbb N$ that keep all but finitely many elements fixed. The QMS associated with length function induced by the non-normalized Hamming metric on $S_\infty$ satisfies $\CGE(\frac 1 2,\infty)$.
\end{example}

Recall that for a countable discrete group $G$, a \emph{F\o lner sequence} is a sequence $\{F_n\}_{n\ge 1}$ of nonempty finite subsets of  $G$ such that 
$$\lim_{n\to\infty}\frac{|gF_n\Delta F_n|}{|F_n|}=0,$$
for every $g\in G$, where $gF=\{gh:h\in F\}$ and $A\Delta B=[A\setminus (A\cap B)]\cup [B\setminus (A\cap B)]$. The group $G$ is called \emph{amenable} if it admits a F\o lner sequence. We refer to \cite[Chapter 2.6]{BO08} for more equivalent definitions and basic properties of amenable groups.
\begin{proposition}
Let $G$ be an amenable group, $\psi\colon G\to[0,\infty)$ a cnd function with associated $1$-cocycle $(H,\pi,b)$. If there exists an orthonormal basis $(e_j)_{j\in J}$ of $H$ such that $\langle b(g),e_j\rangle\in \{0,1\}$ for all $g\in G$, $j\in J$, then the QMS $(\cP_t)$ associated with $\psi$ satisfies $\CGE(1,\infty)$.
\end{proposition}
\begin{proof}
To ease notation, we will only deal with $\GE(1,\infty)$. The proof of complete gradient estimate is similar, embedding $L(G)\otimes \cN$ into a suitable ultraproduct.

Let $(F_n)$ be a F\o lner sequence for $G$ and $\omega\in \beta\mathbb{N}\setminus\mathbb{N}$. Endow $B(\ell^2(F_n))$ with the normalized trace $\tau_n$ and let $p_n$ denote the projection from $\ell^2(G)$ onto $\ell^2(F_n)$. Then we have a trace-preserving embedding
\begin{equation*}
L(G)\to \prod_\omega B(\ell^2(F_n)),\,x\mapsto (p_n x p_n)^\bullet.
\end{equation*}

For each $j$, let $v_j$ be the linear operator on $\ell^2(G)$ given by $v_j (\delta_g) =\langle b(g),e_j\rangle \delta_g$, and denote its restriction to $\ell^2(F_n)$ by the same symbol. Note that for every fixed $n\in\mathbb{N}$, there are only finitely many indices $j\in J$ such that $v_j$ is non-zero on $\ell^2(F_n)$.

Let 
\begin{equation*}
\cH_n=\bigoplus_{j\in J}L^2(B(\ell^2(F_n)),\tau_n)
\end{equation*}
and
\begin{equation*}
\partial_n\colon B(\ell^2(F_n))\to \cH_n,\,a\mapsto ([v_j,a])_j.
\end{equation*}

For $x=\sum_g x_g\lambda_g$ with $\sum_g \psi(g)\abs{x_g}^2<\infty$, we have
\begin{align*}
v_j p_n x p_n (\delta_g)
&=1_{F_n}(g)v_j p_n x (\delta_g)\\
&=\sum_{h\in G}1_{F_n}(g)x_h v_j p_n (\delta_{hg})\\
&=\sum_{h\in G}1_{F_n}(g)1_{F_n}(hg)x_h v_j (\delta_{hg})\\
&=\sum_{h\in G}1_{F_n}(g)1_{F_n}(hg)x_h \langle b(hg),e_j\rangle  \delta_{hg},
\end{align*}
and 
\begin{align*}
p_n x p_n v_j (\delta_g)
&=\langle b(g),e_j\rangle p_n x p_n(\delta_g)\\
&= 1_{F_n}(g) \langle b(g),e_j\rangle p_n x(\delta_{g})\\
&=\sum_{h\in G}1_{F_n}(g)x_h \langle b(g),e_j\rangle p_n (\delta_{hg})\\
&=\sum_{h\in G}1_{F_n}(g)1_{F_n}(hg)x_h \langle b(g),e_j\rangle  \delta_{hg}.
\end{align*}
Hence
\begin{align*}
[v_j,p_n x p_n](\delta_g)&=(v_j p_n x p_n-p_n x p_n v_j) (\delta_g)\\
&=\sum_{h\in G}1_{F_n}(g)1_{F_n}(hg)x_h \langle b(hg)- b(g),e_j\rangle  \delta_{hg},
\end{align*}
and we get
\begin{align*}
&\quad\,\norm{\partial_n (p_n x p_n)}_{\cH_n}^2\\
&=\frac 1{\abs{F_n}}\sum_{g\in F_n}\sum_{j\in J}\langle [v_j,p_n x p_n]\delta_g,[v_j,p_n x p_n]\delta_g\rangle\\
&=\frac 1{\abs{F_n}}\sum_{g\in F_n}\sum_{j\in J}\sum_{h,h'\in G}\bigg(\overline{x_h} x_{h'}\overline{\langle b(hg)-b(g),e_j\rangle} \langle b(h'g)-b(g),e_j\rangle\\
&\qquad\qquad1_{F_n}(hg)1_{F_n}(h'g)\langle \delta_{h g},\delta_{h'g}\rangle\bigg)\\
&=\frac 1{\abs{F_n}}\sum_{g\in F_n}\sum_{j\in J}\sum_{h\in G}|x_h|^2\langle b(hg)-b(g),e_j\rangle|^2 1_{F_n}(hg)\\
&=\frac 1{\abs{F_n}}\sum_{g\in F_n}\sum_{h\in G}|x_h|^2\norm{b(hg)-b(g)}^2 1_{F_n}(hg)\\
&=\frac 1{\abs{F_n}}\sum_{g\in F_n}\sum_{h\in G}\psi(h)|x_h|^2 1_{F_n}(hg)\\
&=\sum_{h\in G}\psi(h)\abs{x_h}^2 \frac {\abs{h^{-1}F_n\cap F_n}}{\abs{F_n}},
\end{align*}
which converges to 
$$\sum_{h\in G}\psi(h)\abs{x_h}^2=\norm{\partial x}^2_{\cH}$$
as $n\to\omega$ by the F\o lner property of $(F_n)$ after an application of the dominated convergence theorem. Thus the tangent bimodule $\cH$ for $(\cP_t)$ can be viewed as a submodule of $\prod_\omega \cH_n$ with the obvious left and right action and $\partial=(\partial_n)^\bullet$.

Let $(\cP_t^n)$ be the QMS on $B(\ell^2(F_n))$ generated by $\partial_n^\dagger\partial_n$. Since $\langle b(\cdot),e_j\rangle$ takes values in $\{0,1\}$, the operators $v_j$'s are projections. Clearly all the $v_j$'s commute. Hence by Theorem \ref{thm:com_proj} and Remark \ref{rmk:extension_group_vna}, $(P_t^n)$ satisfies $\GE(1,\infty)$. From the ultraproduct structure of $\cH$ and $\partial$ we deduce
\begin{align*}
\norm{\partial \cP_t(x_n)^\bullet}_{(\rho_n)^\bullet}^2&=\lim_{n\to\omega}\norm{\partial_n \cP_t^n x_n}_{\rho_n}^2\\
&\leq \lim_{n\to\omega} e^{-2t}\norm{\partial_n x_n}_{\cP_t^n \rho_n}^2\\
&=e^{-2t}\norm{\partial (x_n)^\bullet}_{\cP_t (\rho_n)^\bullet}^2
\end{align*} 
for $(x_n)^\bullet\in L(G)$ and $(\rho_n)^\bullet\in L(G)_+$. In other words, $(\cP_t)$ satisfies $\GE(1,\infty)$.
\end{proof}

\begin{remark}
The group von Neumann algebra embeds into an ultraproduct of matrix algebras if and only if the underlying group is hyperlinear, so it might be possible to extend the previous proposition beyond amenable groups. %On the other hand, the proof seems to depend on the concrete form of this embedding, which only works for amenable groups, so a possible extension seems not straightforward.
\end{remark}

\begin{example}[Amenable groups acting on trees]
Let $\mathcal{T}$ be a tree (viewed as unoriented graph) and $G$ an amenable subgroup of $\mathrm{Aut}(\mathcal{T})$. For fixed $x_0\in \mathcal{T}$ define the length function $\psi$ on $G$ by $\psi(g)=d(x_0,gx_0)$, where $d$ is the combinatorial graph distance. As in the case of free groups, one sees that $\psi$ is conditionally negative definite and the associated $1$-cocycle can be described as follows (see \cite[Example C.2.2]{BHV08}):

Let $E=\{(x,y)\mid x\sim y\}$ be the set of oriented edges of $\mathcal{T}$, and for $e=(x,y)\in E$ write $\bar e=(y,x)$. Let $H=\{\xi\in \ell^2(E)\mid \xi(\bar e)=-\xi(e)\}$ with inner product
\begin{equation*}
\langle \xi,\eta\rangle=\frac 1 2\sum_{e\in E}\xi(e)\eta(e).
\end{equation*}
The action of $G$ on $H$ is given by $\pi(g)\xi(x,y)=\xi(gx,gy)$, and the $1$-cocycle $b$ is given by
\begin{equation*}
b(g)(e)=\begin{cases}1&\text{if }e\text{ lies on }[x_0,gx_0],\\-1&\text{if }\bar e\text{ lies on }[x_0,gx_0],\\0&\text{otherwise},\end{cases}
\end{equation*}
where $[x_0,gx_0]$ denotes the unique geodesic joining $x_0$ and $gx_0$.

Put $F=\{(x,y)\in E\mid d(x_0,x)<d(x_0,y)\}$. Then $(1_e-1_{\bar e})_{e\in F}$ is an orthonormal basis of $H$ and $\langle b(g),1_e-1_{\bar e}\rangle\in \{0,1\}$ for all $g\in G$ and $e\in F$. Thus the QMS associated with $\psi$ satisfies $\CGE(1,\infty)$.

For example this is the case for $G=\mathbb{Z}$ with $\psi(k)=\abs{k}$. Here the tree is the Cayley graph of $\IZ$ and the action is given by the left-regular representation. This QMS on $L(\IZ)$ corresponds, under the Fourier transform, to the Poisson semigroup on $L^\infty(S^1)$.
\end{example}

More generally, the Cayley graph of a group is a tree if and only if it is of the form $\IZ^{\ast k}\ast \IZ_2^{\ast l}$ for $k,l\geq 0$. This group is not amenable unless $k+l\leq 1$, but the free product structure allows us to obtain the same bound.

\begin{theorem}
If $G$ is a group whose Cayley graph is a tree and the cnd function $\psi$ is given by $\psi(g)=d(g,e)$, where $d$ is the combinatorial metric on the Cayley graph, then the QMS associated with $\psi$ satisfies $\CGE(1,\infty)$ and $\CLSI(2)$ and the constants in both inequalities are optimal.
\end{theorem}
\begin{proof}
As previously mentioned, $G$ is of the form $\IZ^{\ast k}\ast \IZ_2^{\ast l}$ with $k,l\geq 0$. It is not hard to see that the QMS associated with $\psi$ decomposes as free product of the QMS associated with the word length functions on the factors. Thus it satisfies $\CGE(1,\infty)$ by Theorem \ref{thm:free_product} and $\CLSI(2)$ by Corollary \ref{cor:MLSI}. Since both the gradient estimate and the modified logarithmic Sobolev inequality imply that the generator has a spectral gap of $1$, the constants are optimal.
\end{proof}

\begin{example}
If $G$ is a free group and $\psi$ the word length function, then the associated QMS satisfies $\CGE(1,\infty)$ and $\CLSI(2)$. Note that the usual logarithmic Sobolev inequality, which is equivalent to the optimal hypercontractivity, is still open. Some partial results have been obtained in \cite{JPPR15,RX16}. Our optimal modified LSI supports the validity of optimal LSI from another perspective.
\end{example}

\newcommand{\etalchar}[1]{$^{#1}$}

\end{document}